\documentclass[11pt,reqno]{amsart}

\usepackage[dvipsnames]{xcolor}
\usepackage{tikz}
\usepackage{stmaryrd}
\usetikzlibrary{matrix,arrows,decorations.pathmorphing}
\usepackage{graphicx}
\usepackage{amssymb}
 \usepackage[stable]{footmisc}
\usepackage{amsmath,mathtools}
\usepackage{fullpage}
\usepackage{caption}
\usepackage{amsthm}
\usepackage{mathrsfs}
\usepackage{thmtools}
\usepackage{blkarray}
\usepackage{multirow}
\usepackage{picinpar} 
\usepackage{tikz-cd}
\usepackage{color}
\usepackage{verbatim}
\usepackage{hyperref}
\hypersetup{colorlinks=true, linkcolor=black, citecolor = OliveGreen, urlcolor = black}
\usepackage{amssymb}
\usepackage{amsmath}
\usepackage{enumitem,colonequals}

\usepackage{color}
\usepackage{mathrsfs}
\usepackage[all]{xy}
\usepackage{comment}
\usepackage{mathpazo}
\usepackage{euler}
\usepackage[noabbrev,capitalise]{cleveref}

\usepackage[margin=1.in]{geometry}
\tikzset{
  closed/.style = {decoration = {markings, mark = at position 0.5 with { \node[transform shape, xscale = .8, yscale=.4] {/}; } }, postaction = {decorate} },
  open/.style = {decoration = {markings, mark = at position 0.5 with { \node[transform shape, scale = .7] {$\circ$}; } }, postaction = {decorate} }
}

\newcommand{\mZ}{\mathbb{Z}}

\newcommand{\mQ}{\mathbb{Q}}

\newcommand{\mC}{\mathbb{C}}
\newcommand{\mG}{\mathbb{G}}

\newcommand{\mA}{\mathbb{A}}

\newcommand{\mP}{\mathbb{P}}

\newcommand{\wt}[1]{\widetilde{#1}}

\usepackage[OT2,T1]{fontenc}
\DeclareSymbolFont{cyrletters}{OT2}{wncyr}{m}{n}
\DeclareMathSymbol{\Sha}{\mathalpha}{cyrletters}{"58}

\DeclareMathSymbol{\Sha}{\mathalpha}{cyrletters}{"58}

\newcommand{\brk}[1]{ \left\lbrace #1 \right\rbrace }

\newcommand{\sB}{{\mathscr B}}
\newcommand{\sC}{{\mathscr C}}
\newcommand{\sD}{{\mathscr D}}
\newcommand{\sE}{{\mathscr E}}

\newcommand{\sG}{{\mathscr G}}

\newcommand{\sI}{{\mathscr I}}
\newcommand{\sJ}{{\mathscr J}}

\newcommand{\sL}{{\mathscr L}}
\newcommand{\sM}{{\mathscr M}}

\newcommand{\sO}{{\mathscr O}}

\newcommand{\sS}{{\mathscr{S}}}
\newcommand{\sT}{{\mathscr T}}
\newcommand{\sU}{{\mathscr U}}
\newcommand{\sV}{{\mathscr V}}
\newcommand{\sW}{{\mathscr W}}
\newcommand{\sX}{{\mathscr X}}
\newcommand{\sY}{{\mathscr Y}}
\newcommand{\sZ}{{\mathscr Z}}

\RequirePackage{mathrsfs} 

\numberwithin{equation}{subsection}
\newtheorem{thmx}{Theorem}

\newtheorem{corox}[thmx]{Corollary}

\numberwithin{equation}{subsection}

\newtheorem{theorem}[subsection]{Theorem}

\newtheorem{lemma}[subsection]{Lemma}

\newtheorem{corollary}[subsection]{Corollary}
\newtheorem{conjecture}[subsection]{Conjecture}
\newtheorem{prop}[subsection]{Proposition}

\theoremstyle{definition}
\newtheorem{definition}[subsection]{Definition}
\newtheorem{example}[subsection]{Example}

\theoremstyle{remark}

\newtheorem{remark}[subsection]{Remark}

\numberwithin{equation}{section} \numberwithin{figure}{section}

\DeclareMathOperator{\an}{an}
\DeclareMathOperator{\red}{red}
\DeclareMathOperator{\Ber}{Ber}
\DeclareMathOperator{\alg}{alg}
\DeclareMathOperator{\topo}{top}
\DeclareMathOperator{\Pic}{Pic} 
 
 \DeclareMathOperator{\Spec}{Spec}

\DeclareMathOperator{\Spa}{Spa}

\DeclareMathOperator{\Int}{Int}
\DeclareMathOperator{\Stab}{Stab}

\DeclareMathOperator{\Sym}{Sym}

\DeclareMathOperator{\pr}{pr}

\DeclareMathOperator{\Sp}{Sp}

\DeclareMathOperator{\ett}{\acute{e}t}

\DeclareMathOperator{\temp}{temp}

\newcommand{\tr}{\mathrm{tr}}

\newcommand{\cdef}[1]{\textsf{\textit{#1}}}

\newcommand\fm{\mathfrak{f}}

\renewcommand{\leq}{\leqslant}

\renewcommand{\geq}{\geqslant}

\DeclareMathOperator{\QAlb}{QAlb}
\DeclareMathOperator{\codim}{codim}

\makeatletter
\@namedef{subjclassname@1991}{\emph{2020} Mathematics Subject Classification}
\makeatother

\begin{document}

\title{A non-Archimedean analogue of Campana's notion of specialness}	

\author{Jackson S. Morrow}
\address{Jackson S. Morrow \\
	Centre de Recherche de Math\'ematiques, Universit\'e de Montr\'eal\\
    Montr\'eal, Qu\'ebec H3T 1J4\\
    CAN}
\email{jmorrow4692@gmail.com}
\urladdr{\url{https://sites.google.com/site/jacksonsalvatoremorrow/}} 

\author{Giovanni Rosso}
\address{Giovanni Rosso \\
	Concordia University \\
	Department of Mathematics and Statistics \\ 
	Montr\'eal, Qu\'ebec H3G 1M8 \\
	CAN
}
\email{giovanni.rosso@concordia.ca}
\urladdr{\url{https://sites.google.com/site/gvnros/}}


\date{\today}

\begin{abstract}
Let $K$ be an algebraically closed, complete, non-Archimedean valued field of characteristic zero, and let $\mathscr{X}$ be a $K$-analytic space (in the sense of  Huber). 
In this work, we pursue a non-Archimedean characterization of Campana's notion of specialness. 
We say $\mathscr{X}$ is $K$-analytically special if there exists a connected, finite type algebraic group $G/K$, a dense open subset $\mathscr{U}\subset G^{\text{an}}$ with $\codim(G^{\text{an}}\setminus \mathscr{U}) \geq 2$, and an analytic morphism $\mathscr{U} \to \mathscr{X}$ which is Zariski dense.

With this definition, we prove several results which illustrate that this definition correctly captures Campana's notion of specialness in the non-Archimedean setting. 
These results inspire us to make non-Archimedean counterparts to conjectures of Campana. 
As preparation for our proofs, we prove auxiliary results concerning the indeterminacy locus of a meromorphic mapping between $K$-analytic spaces, the notion of pseudo-$K$-analytically Brody hyperbolic, and extensions of meromorphic maps from smooth, irreducible  $K$-analytic spaces to the analytification of a semi-abelian variety.
\end{abstract}
\maketitle

\section{\bf Introduction}
\label{sec:intro}

The study of rational points on varieties over number fields is one of the fundamental questions in arithmetic geometry. 
For curves, we have a good understanding of the behaviour of the rational points, and this behaviour is governed by the genus of the curve. 
If the curve has genus $g\leq 1$, then the rational points on the curve become infinite after a finite extension of the base number field; in this case, we say the curve is arithmetically special. 
If the curve has genus $g\geq 2$, then a famous result of Faltings \cite{Faltings2} asserts that the rational points on the curve are finite over any number field, and in this case, we say the curve is arithmetically hyperbolic. 
It is natural to ask if a similar dichotomy holds for higher dimensional varieties. 

The deep and contrasting conjectures of Lang \cite[p.~161--162]{Lang} and Campana \cite[Section 9]{campana2004orbifolds} posit that such a relationship holds for certain classes of varieties. 
Lang's conjecture asserts that a variety of general type over a number field is pseudo-arithmetically hyperbolic (\cite[Definition 7.2]{JBook}), and Campana's conjecture claims that a special variety (\autoref{def:specialBogomolov}) over a number field is arithmetically special (\autoref{defn:arithmeticallyspecial}).
There has been significant progress on both of these conjectures (see e.g., \cite{FaltingsLang1, FaltingsLang2, Vojta11,Vojta:IntegralPointsII} and \cite{HarrisTsch, HassetTschinkel:AbelianFibrations, bogomolov2000density} and the excellent book \cite{nicole2020arithmetic} for more references). 
Since understanding the arithmetic properties of such varieties is difficult, we seek other ways to describe being of general type and special.
There exist (conjectural) complex analytic characterizations of these notions (see e.g., \cite{Lang, Kobayashi} and \cite{campana2004orbifolds}), and recently, there has been work on providing a (conjectural) non-Archimedean characterization of general type (see e.g., \cite{Cherry, CherryKoba, JVez, MorrowNonArchGGLV, sun:NonArchBorelhyperbolic}).

\subsection*{Main contributions}
In this work, we offer a non-Archimedean interpretation of Campana's notion of specialness with the desideratum that our interpretation is equivalent to other characterizations of specialness.  
Our definition is motivated by the (conjectural) complex analytic analogue of specialness, which goes by the name Brody special (\autoref{defn:Brodyspecial}) and states that a complex manifold is Brody special if there exists a dense entire map from $\mC$ to it. For example, the complex analytification of an abelian variety is Brody special by the Riemann uniformization theorem.

A natural first guess for a non-Archimedean analogue of this notion would be to ask for a dense analytic morphism from non-Archimedean analytification of $\mA^1$ or $\mG_m$ into our analytic space. However, results of Cherry \cite{Cherry} tell us that is definition will not suffice. For example, the non-Archimedean analytification of an abelian variety with good reduction will not admit any non-constant morphism from the these spaces, and since these analytic spaces should be special we see that this naive notion does not suffice.

Instead of testing specialness on non-Archimedean entire curves, we will test on big analytic opens of the non-Archimedean analytification of connected algebraic groups. 
To set notation for our definition, let $K$ be an algebraically closed, complete, non-Archimedean valued field of characteristic zero, let $\mathscr{X}$ be a $K$-analytic space (in the sense of Huber), and let $X^{\an}$ denote the non-Archimedean analytification of a variety $X$ over $K$. 

\begin{definition}
We say that $\sX$ is \cdef{$K$-analytically special} if there exists a connected, finite type algebraic group $G/K$, a dense open subset $\sU\subset G^{\an}$ with $\codim(G^{\an}\setminus \sU) \geq 2$, and an analytic morphism $\sU \to \sX$ which is Zariski dense. 
\end{definition}

We highlight two points in the above definition. 
First, it may seem unnatural to test whether a $K$-analytic space is $K$-analytically special by asking for the existence of a Zariski dense analytic morphism a big analytic open of an algebraic group, however we provide explanations for these conditions in \autoref{exam:whybig} and \autoref{remark:whybiganalytic}. 
The second is that we do not require our $K$-analytic space to be compact. This allows us to avoid the language of logarithmic geometry and orbifolds when working with non-compact $K$-analytic space, which eliminates some technical difficulties.

We prove several results to illustrate that our definition correctly captures Campana's notion of specialness (\autoref{def:specialBogomolov} and \autoref{def:specialfibration}). 
Our first result states that a  $K$-analytically special $K$-analytic space cannot dominate a positive dimensional pseudo-$K$-analytically Brody hyperbolic $K$-analytic space, which can be viewed as a non-Archimedean version of \cite[Proposition 9.27]{campana2004orbifolds}. 
We refer the reader to Section \ref{sec:Brodyhyperbolic} for the definition and a discussion of the notion of pseudo-$K$-analytically Brody hyperbolic. 

\begin{thmx}\label{thm:nofibrationtoBrody}
Let $K$ be an algebraically closed, complete, non-Archimedean valued field of characteristic zero, and let $\sX,\sY$ be irreducible, reduced, separated $K$-analytic spaces. 
If $\sY$ is a  $K$-analytically special $K$-analytic space and $\sX$ is a positive dimensional pseudo-$K$-analytically Brody hyperbolic $K$-analytic space (\autoref{defn:KBrodyMod}), then there is no dominant morphism $\sY \to \sX$.
\end{thmx}

Using \autoref{thm:nofibrationtoBrody}, we are able to identify several classes of $K$-analytically special $K$-analytic spaces in terms of their intrinsic geometry and certain abelian properties of their fundamental group. 
We recall that our guiding principle is that the notions of specialness and being of general type contrast each other.

In \cite{MorrowNonArchGGLV}, the first author proved that a closed subvariety $X$ of a semi-abelian variety over $K$ is of logarithmic general type if and only if it is pseudo-$K$-analytically Brody hyperbolic. This result builds off of the works \cite{Abram, Nogu}, where the authors show that the first condition is equivalent to the special locus of $X$ being properly contained in $X$, which essentially means that $X$ is not the translate of a semi-abelian subvariety. 
With our guiding principle in mind, we use results of Vojta and \autoref{thm:nofibrationtoBrody} to prove that $X$ being a translate of a semi-abelian subvariety is equivalent to $X^{\an}$ being $K$-analytically special. 

\begin{thmx}\label{thm:closedabelianspecial}
Let $K$ be an algebraically closed, complete, non-Archimedean valued field of characteristic zero. 
Let $G/K$ be a semi-abelian variety, and let $X\subset G$ be a closed subvariety. 
Then, $X$ is the translate of a semi-abelian subvariety if and only if $X^{\an}$ is $K$-analytically special. 
\end{thmx}

We point out that the analogous result for Brody special (i.e., for complex analytic varieties) holds due to work of  \cite{Abram, Nogu, campana2004orbifolds}. 
We also note that results of Iitaka \cite[Theorems 2 \& 4]{IitakaLogForms} (see also Vojta's result in \cite[Theorem 5.15]{Vojta:IntegralPointsII}) and Campana \cite[Theorem 5.1]{campana2004orbifolds} imply that if $X$ is a translate of a semi-abelian variety, then $X$ is a special variety (see \autoref{def:specialfibrationlog} for the notion of special non-proper variety). The techniques of the proof of \autoref{thm:closedabelianspecial} can easily be adapted to the algebraic case, thus giving us the opposite implication of this result, see \autoref{thm:specialsemiabelian}.
We record such a result as an immediate corollary to \autoref{thm:closedabelianspecial}. 

\begin{corox}\label{coro:specialabelian}
Let $K$ be an algebraically closed, complete, non-Archimedean valued field of characteristic zero. 
Let $G/K$ be a semi-abelian variety, and let $X\subset G$ be a closed subvariety. 
Then, $X$ is special if and only if $X^{\an}$ is $K$-analytically special. 
\end{corox}

Our final main result is related to the abelianity conjecture of Campana \cite[Conjecture 7.1]{campana2004orbifolds}. 
This conjecture postulates that for a variety $X/\mC$, being special is equivalent to the topological fundamental group of the complex analytification of $X$ being virtually abelian (i.e., it contains a finite index abelian subgroup). 
For results concerning this abelianity conjecture, we refer the reader to \cite[Theorem 7.8]{campana2004orbifolds}, \cite[Theorem 1.1]{yamanoi2010fundamental}, and \cite[Theorem 1.12]{javanpeykar2020albanese}.

The development of fundamental groups for a non-Archimedean analytic space has a rich history, which we briefly exposit. 
Berkovich $K$-analytic spaces posses nice topological properties. 
For example, an important result of Berkovich \cite[Corollary 9.6]{BerkovichUniversalCover} says that a smooth, connected Berkovich $K$-analytic space admits a topological universal covering which is a simply connected $K$-analytic space, and hence one can describe the topological fundamental group of a Berkovich $K$-analytic space via loops modulo homotopy. 
This illustrates that Berkovich $K$-analytic spaces have similarities to complex manifolds, however we note that the topological fundamental group does not detect interesting arithmetic properties of a variety; indeed, the topological fundamental group of the analytification of any variety over $K$ with good reduction is trivial.

While there are too few topological coverings, working directly with the \'etale fundamental group of a Berkovich $K$-analytic space is unwieldly (see e.g., \cite[Proposition 7.4]{DeJongFundamentalGroupNonArch}).  
To remedy this, Andr\'e \cite{andre-per} introduced the tempered fundamental group of a Berkovich $K$-analytic space, which sits between the topological and \'etale fundamental groups and provides us with the correct fundamental group to study non-Archimedean analogues of Campana's abelianity conjectures. 

As a first step in this direction, we prove that the Berkovich analytification of a projective surface with negative Kodaira dimension is $K$-analytically special if and only if its tempered fundamental group is virtually abelian; see \cite[Theorem 3.1]{buzzard2000algebraic} for the analogous complex analytic statement. 

\begin{thmx}\label{thm:surfaceKinfinity}
Let $K$ be an algebraically closed, complete, non-Archimedean valued field of characteristic zero, and let $X/K$ is a smooth, projective surface with negative Kodaira dimension. 
Then the following are equivalent:
\begin{enumerate}
\item $X$ has irregularity $q(X) = h^0(X,\Omega_X^1)$ less than 2;
\item $X^{\an}$ is $K$-analytically special;
\item the tempered fundamental group $\pi_1^{\temp}(X^{\an})$ of $X^{\an}$ is virtually abelian.
\end{enumerate}
\end{thmx}

The above results inspire us to formulate non-Archimedean counterparts to conjectures of Campana. 
Below, we present shortened versions of these conjectures as many of the definitions have been omitted from our discussion up to this point. For precise statements of these conjectures, we refer the reader to Section \ref{sec:NonArchCampana}. 

\begin{conjecture}[Non-Archimedean Campana's conjectures]
Let $X/K$ be a smooth projective variety. Then, the following are equivalent:
\begin{enumerate}
\item $X$ is special (\autoref{def:specialBogomolov}, \autoref{def:specialfibration});
\item $X^{\an}$ is $k$-analytically special. 
\end{enumerate}
\end{conjecture}

\begin{conjecture}[Non-Archimedean Campana's abelianity conjecture for fundamental groups]
Let $X/K$ be a smooth projective variety. If $X^{\an}$ is $K$-analytically special, then $\pi_1^{\temp}(X^{\an})$ is virtually abelian.
\end{conjecture}

\subsection*{Preparatory results}
In order to prove the above theorems, we need to prove three auxiliary results, which we believe to be of independent interest.

The first one (\autoref{thm:meromorphicmap}) concerns the indeterminacy locus of a meromorphic mapping between $K$-analytic spaces and is a non-Archimedean analogue of a result due to Remmert \cite{remmert1957holomorphic}. The proof presented in that section has been provided to us by Brian Conrad. 
We use this result to show that our notion of $K$-analytically special is a bi-meromorphic invariant, which is a crucial property (see e.g., the proof of \autoref{thm:surfaceKinfinity}). 
Roughly speaking, a $K$-analytic space being $K$-analytically special means that it admits a dominant, meromorphic map from an algebraic group, and so in order for this notion to be a bi-meromorphic invariant, we need to understand the indeterminancy locus of a general meromorphic map. 
\autoref{thm:meromorphicmap} states that when the source is normal and the target is reduced and proper, the domain of definition of a meromorphic map is an open \textit{analytic} subset whose complement has codimension at least two. Note that this is precisely is the condition we require in our definition of $K$-analytically special.

For the second result, we begin by offering a new, more natural definition of pseudo-$K$-analytically Brody hyperbolic (\autoref{defn:KBrodyMod}) and deduce an equivalent way of testing this notion (\autoref{thm:equivalentBrody}), which brings it closer in line with our notion of $K$-analytically special.  
To expand on this, the first definition of pseudo-$K$-analytically Brody hyperbolic appeared in \cite[Definition 2.2]{MorrowNonArchGGLV} and contained a seemingly unnatural condition of studying algebraic maps from big algebraic opens of abelian varieties. 
Moreover, with this original definition, it is unclear if the statement of \autoref{thm:nofibrationtoBrody} is true. 
To fix this issue, we modify this definition to test pseudo-$K$-analytically Brody hyperbolicity on big \textit{analytic} opens of analytifications of algebraic groups and prove that one can actually test this notion on analytic maps from $\mG_{m,K}^{\an}$ and from big \textit{analytic} opens of analytifications of abelian varieties. 
With this new definition, we immediately arrive at \autoref{thm:nofibrationtoBrody}. 

The final preparatory theorem (\autoref{thm:extendanalyticmorphism}) is an extension result concerning meromorphic maps from smooth, irreducible adic spaces to the analytification of a semi-abelian variety, which is a non-Archimedean analogue of \cite[Section 8.4, Corollary 6]{BLR} and \cite[Lemma A.2]{MochizukiAbsoluteAnabelian}. 
We use this result and \autoref{thm:equivalentBrody} show that \cite[Theorem A]{MorrowNonArchGGLV} remains true with our new definition of pseudo-$K$-analytically Brody hyperbolic (see \autoref{prop:Brodyequivalence}). 
Equipped with this fact, the proof of \autoref{thm:closedabelianspecial} follows from utilizing results of Vojta \cite{Vojta:IntegralPointsII} on the Ueno fibration of a closed subvariety of a semi-abelian variety.

\subsection*{Organization}
The paper has two parts. Sections \ref{sec:nonArchspaces}, \ref{sec:meromorphicmaps}, and \ref{sec:Brodyhyperbolic} form the first part and consist of background material and auxiliary results. The remaining sections focus on defining and proving results related to our notion of $K$-analytically special. More precisely, we organize the paper as follows.

In Section \ref{sec:nonArchspaces}, we review non-Archimedean analytic spaces, describe several properties preserved by analytification, and study the sheaf of meromorphic functions on a taut locally strongly Noetherian adic space. 
In Section \ref{sec:meromorphicmaps}, we prove \autoref{thm:meromorphicmap}, following the proof given to us by Brian Conrad, which is a non-Archimedean analogue of a theorem of Remmert stating that the indeterminacy locus of a meromorphic morphism between a normal and a proper non-Archimedean analytic space has codimension at least two. 
In Section \ref{sec:Brodyhyperbolic}, we give a new definition of pseudo-$K$-analytically Brody hyperbolic and deduce an equivalent way of testing this notion (\autoref{thm:equivalentBrody}).

We begin our discussion on the various notions of specialness in Section \ref{sec:specialnotions} and offer our non-Archimedean characterization of specialness, which we call $K$-analytically special, in Section \ref{sec:nonArchspecial}. In this section, we prove several basic properties of being $K$-analytically special and our first main theorem (\autoref{thm:nofibrationtoBrody}). 
In Section \ref{sec:specialsubvarietiessemiabelian}, we prove our second main theorem (\autoref{thm:closedabelianspecial}) concerning when a closed subvariety of a semi-abelian variety is $K$-analytically special. 
In Section \ref{sec:sufacenegKodaira}, we prove our final main theorem (\autoref{thm:surfaceKinfinity}) which characterizes when a projective surface of negative Kodaira dimension is $K$-analytically special in terms of its tempered fundamental group. 
To conclude, we make non-Archimedean counterparts to Campana's conjectures in Section \ref{sec:NonArchCampana}. 

\subsection*{Conventions}
We establish the following to be used throughout. 

\subsubsection*{Fields and algebraic geometry}
We will let $k$ denote an algebraically closed field of characteristic zero and let $K$ be an algebraically closed, complete, non-Archimedean valued field of characteristic zero. A variety $X$ over a field will be an integral, separated scheme of finite type over said field. 
We will use $K_X$ to denote the canonical divisor on $X$ and $q(X) = h^0(X,\Omega_X^1)$ to denote the irregularity of $X$. 
For a smooth variety $X/k$ and a line bundle $\mathscr{L}$ on $X$, we use $\kappa(X,\mathscr{L})$ to denote the Iitaka dimension of $\mathscr{L}$, which we briefly recall. 
For each $m\geq 0$ such that $h^0(X,\mathscr{L}^{\otimes m}) \neq 0$, the linear system $|\mathscr{L}^{\otimes m}|$ induces a rational map from $X$ to a projective space of dimension $h^0(X,\mathscr{L}^{\otimes m}) - 1$. 
The Iitaka dimension of $\mathscr{L}$ is the maximum over all $m\geq 1$ of the dimension of the image for this rational map. 
The Kodaira dimension $\kappa(X)$ of $X$ is the Iitaka dimension of the canonical bundle. 

\subsubsection*{Analytic spaces}
We will make use of various analytifications of a variety $X$ over a field.

When $k = \mC$, we will use $X^{\an}$ to denote the complex analytification of $X$, which corresponds to taking the complex valued points $X(\mC)$ of $X$.

When $k = K$, we will denote the adic space associated with $X$ by $X^{\an}$ (as in \cite{huber2}). Sometimes we will need to consider the corresponding Berkovich space and we denote the Berkovich analytification of $X$ by $X^{\Ber}$ (as in \cite{BerkovichSpectral} or good Berkovich $K$-analytic spaces in \cite{BerkovichEtaleCohomology}).  
As these non-Archimedean analytifications are fundamental objects in our study, we devote Section \ref{sec:nonArchspaces} to describing their properties and relationships. 
When referring to rigid analytic, Berkovich, and adic spaces which may not be algebraic, we will use script notation $\sX,\sY,\sZ$. 
In our work, all rigid analytic spaces over $\Sp(K)$ are taut, all Berkovich $K$-analytic spaces are strict and Hausdorff, and all adic spaces are taut, locally strongly Noetherian, and locally of finite type over $\Spa(K,K^{\circ})$, unless otherwise stated. 
We will mainly use adic spaces but sometimes we need to refer to a different type of analytic space and will make it clear as to which category of analytic space we are using in certain instances.

That being said, whenever we refer to a \cdef{$K$-analytic space}, we mean a taut, locally strongly Noetherian, and locally of finite type adic space over $\Spa(K,K^{\circ})$.

\subsection*{Acknowledgments}
We are very grateful to Brian Conrad for supplying the proof of \autoref{thm:meromorphicmap} and to Ariyan Javanpeykar for the suggestion to look at big opens for the definition of $K$-analytically special. 
We also thank Marta Benozzo, Marc-Hubert Nicole, and Remy van Dobben de Bruyn for helpful conversations and extend our thanks to Lea Beneish, Ariyan Javanpeykar, Marc-Hubert Nicole, and Alberto Vezzani for useful comments on a first draft.

\section{\bf Preliminaries on non-Archimedean analytic spaces}
\label{sec:nonArchspaces}

In this section, we provide necessary background on non-Archimedean analytic spaces. 
In particular, we describe the equivalence between certain subcategories of rigid analytic, Berkovich $K$-analytic, and adic spaces, recall basic properties of adic spaces, discuss properties of analytifications of algebraic varieties and algebraic morphisms, and finally introduce the sheaf of meromorphic functions on a taut locally strongly Noetherian adic space.

\subsection{Comparisons between rigid analytic, Berkovich $K$-analytic spaces, and adic spaces}\label{subsec:comparision}
To begin, we recall the comparision between rigid analytic, Berkovich $K$-analytic, and adic spaces. 

\begin{theorem}[\protect{\cite[Theorem 1.6.1]{BerkovichEtaleCohomology}, \cite[Proposition 8.3.1]{huber}}]\label{thm:comparisionrigidadic}
The category of taut rigid analytic spaces over $\Sp(K)$ is equivalent to the category of Hausdorff strict Berkovich $K$-analytic spaces. 
\end{theorem}

\begin{theorem}[\protect{\cite[Proposition 8.3.1]{huber}, \cite[Theorem 0.1]{henkel:comparision}}]\label{thm:comparision}
The category of taut adic spaces locally of finite type over $\Spa(K,K^{\circ})$ is equivalent to the category of Hausdorff strict Berkovich $K$-analytic spaces. 
At the level of topological spaces, this equivalence sends an adic space $\sX$ to its universal Hausdorff quotient $[\sX]$. 
\end{theorem}

\begin{remark}\label{remark:nonHausdorff}
We note that the adic space associated with an algebraic variety is an example of a taut adic space locally of finite type, and since it will be relevant in Subsection \ref{appendix}, we emphasize that \autoref{thm:comparision} tells us that taut adic spaces locally of finite type are not Hausdorff. 
\end{remark}

We also recall that the notions of properness and finiteness are equivalent for rigid analytic and Berkovich $K$-analytic spaces.

\begin{lemma}\label{lemma:equivalenceproperfinite}
Let $\sX,\sY$ be rigid analytic spaces over $K$, and let $\sX^{\Ber}$, $\sY^{\Ber}$ denote the associated Berkovich $K$-analytic spaces. 
Let $f\colon \sX \to \sY$ denote a morphism of rigid analytic spaces, and let $f^{\Ber}\colon \sX^{\Ber} \to \sY^{\Ber}$ denote the associated morphism of Berkovich $K$-analytic spaces. Then,
\begin{enumerate}
\item $f$ is proper if and only if $f^{\Ber}$ is proper,
\item $f$ is finite if and only if $f^{\Ber}$ is finite. 
\end{enumerate}
\end{lemma}

\begin{proof}
This follows from \cite[Proposition 3.3.2]{BerkovichSpectral} and \cite[Example 1.5.3]{BerkovichEtaleCohomology}. 
\end{proof}

\subsection{Basic properties of adic spaces}
We now recall the notion of reduced, normal, and irreducible adic spaces following \cite{mann:propertiesadicspaces}.

\begin{definition}[\protect{\cite[Definition 1.3]{mann:propertiesadicspaces}}]\label{defn:normal}
The adic space $\sX$ is \cdef{normal} (resp.~\cdef{reduced}) if it can be covered by affinoid adic spaces of the form $\Spa(A,A^{+})$ where $A$ is normal (resp.~reduced). 
\end{definition}

%

\begin{definition}[\protect{\cite[Definition 1.11]{mann:propertiesadicspaces}}]
The adic space $\sX$ is \cdef{irreducible} if it cannot be written as the disjoint union of two proper closed adic subspaces.
\end{definition}

\subsection{Properties of analytifications of algebraic morphisms}\label{subsec:propertiesanalytification}
We now discuss facts concerning the analytification functor from locally finite type $K$-schemes to adic spaces over $K$. 
The analogous results for rigid analytic spaces over $K$ and Berkovich $K$-analytic are treated in \cite[Section 5]{conrad-conn} and \cite[Section 3.4]{BerkovichSpectral}, respectively. 

\begin{lemma}\label{lemma:absoluteproperties}
Let $X$ be a $K$-scheme which is locally of finite type, and let $X^{\an}$ (resp.~$X^{\Ber}$) denote the adic space  (resp.~Berkovich $K$-analytic space) associated with $X$. 
Then,
\begin{enumerate}
\item $X$ is reduced if and only if $X^{\an}$ (resp,~$X^{\Ber}$) is reduced,
\item $X$ is normal if and only if $X^{\an}$ (resp.~$X^{\Ber}$) is normal,
\item $X$ is separated if and only if $X^{\an}$ is separated (resp.~$|X^{\Ber}|$ is Hausdorff),
\item $X$ is smooth if and only if $X^{\an}$ (resp.~$X^{\Ber}$) is smooth,
\item $X$ is irreducible if and only if $X^{\an}$ (resp.~$X^{\Ber}$) is irreducible. 
\end{enumerate}
\end{lemma}

\begin{proof}
The first and second statements follow from \cite[Proposition 3.4.3]{BerkovichSpectral} for Berkovich $K$-analytic spaces and from \cite[Theorem 5.1.3.(1)]{conrad-conn} and \cite[\S 1.1.11.(c)]{huber} for adic spaces. 
The third statement follows from \cite[Theorem 3.4.8.(1)]{BerkovichSpectral} for Berkovich $K$-analytic spaces and from \cite[Theorem 5.2.1]{conrad-conn} and \cite[Remark 1.3.19]{huber} for adic spaces. 
The fourth statement follows from \cite[Proposition 3.4.6.(3)]{BerkovichSpectral} for Berkovich $K$-analytic spaces and from \cite[Theorem 5.2.1]{conrad-conn} and \cite[Proposition 1.7.11]{huber} for adic spaces. 
The fifth statement follows from  \cite[Proposition 2.7.16]{Ducros:Families} for Berkovich $K$-analytic spaces and from \cite[Theorem 2.3.1]{conrad-conn} and \cite[\S 1.1.11.(c)]{huber} for adic spaces. 
We mention that irreducibility of a Berkovich $K$-analytic space refers to irreducibility in the Zariski analytic topology (see \cite[\S 1.5.1]{Ducros:Families} for the definition). 
\end{proof}

\begin{lemma}\label{lemma:morphismpartiallyproper}
Let $f\colon X\to Y$ be a locally of finite type morphism of varieties over $K$, and let $f^{\an}\colon X^{\an} \to Y^{\an}$ (resp.~$f^{\Ber}\colon X^{\Ber} \to Y^{\Ber}$) denote the associated morphism of   adic spaces (resp.~Berkovich $K$-analytic spaces). Then, $f^{\an}$ (resp.~$f^{\Ber}$) is partially proper  (resp.~boundaryless). 
\end{lemma}

\begin{proof}
The second statement  follows from \cite[Proposition 1.5.5.(ii)]{BerkovichEtaleCohomology} since $X^{\Ber}$ and $Y^{\Ber}$ are boundaryless Berkovich $K$-analytic spaces. 

For the first statement we first note that the first statement and \cite[Section 1.6]{BerkovichEtaleCohomology} imply that the associated morphism of rigid analytic spaces is partially proper, and now the result follows from \cite[Remark 1.3.19.(iii)]{huber}. 
\end{proof}

\begin{lemma}\label{lemma:surjectiveflat}
Let $f\colon X\to Y$ be a locally of finite type morphism of varieties over $K$, and let $f^{\an}\colon X^{\an} \to Y^{\an}$ (resp.~$f^{\Ber}\colon X^{\Ber} \to Y^{\Ber}$) denote the associated morphism of adic spaces  (resp.~Berkovich $K$-analytic spaces). Then,
\begin{enumerate}
\item $f$ is surjective if and only if $f^{\an}$ (resp.~$f^{\Ber}$) is surjective, 
\item $f$ is flat if and only if $f^{\an}$ (resp.~$f^{\Ber}$) is flat.
\end{enumerate}
\end{lemma}

\begin{proof}
The first (resp.~second) statement for $f^{\an}$ follows from \autoref{lemma:morphismpartiallyproper} and \cite[p.~487, part (b)]{huber} (resp.~\cite[Lemma (1.1.10.(iii)]{huber}). The statements for $f^{\Ber}$ can be found in \cite[Proposition 3.4.6]{BerkovichSpectral}. 
\end{proof}

\begin{lemma}\label{lemma:flatopen}
Let $f\colon X\to Y$ be a locally of finite type morphism of varieties over $K$, and let $f^{\an}\colon X^{\an} \to Y^{\an}$ (resp.~$f^{\Ber}\colon X^{\Ber} \to Y^{\Ber}$) denote the associated morphism of adic spaces  (resp.~Berkovich $K$-analytic spaces). 
If $f$ is flat, then $f^{\an}$ (resp.~$f^{\Ber}$) is flat and partially proper  (resp.~boundaryless) and open. 
\end{lemma}

\begin{proof}
This follows from \autoref{lemma:morphismpartiallyproper}, \autoref{lemma:surjectiveflat}.(2), and \cite[Theorem 9.2.3 and Remark 9.2.4]{Ducros:Families} and \cite[p.~425]{huber}. 
\end{proof}

\subsection{The sheaf of meromorphic functions on an adic space}
To conclude this preliminaries section, we define the sheaf of meromorphic functions on a taut, locally strongly Noetherian adic space over $\Spa(K,K^{\circ})$. 
We follow the exposition of Bosch \cite{bosch1982meromorphic}. 

First, we define the sheaf locally for a strongly Noetherian affinoid adic space.  
For any ring $R$, denote by $Q(R)$ its total ring of fractions. 
Let $\sX = \Spa(A,A^{+})$ be an affinoid adic space where $A$ is strongly Noetherian. 
Let $\sU \subset \sX$ be a rational subset. 
By \cite[(II.I.IV)]{huber2}, the restriction map $\sO_{\sX}(\sX) \to \sO_{\sX}(\sU)$ is flat, and hence it gives us a homomorphism between the corresponding total rings of fractions. 
Moreover, we can define a presheaf $\sM_{\sX}$ on the set of rational subset $\sU$ of $\sX$ via
\[
\sM_{\sX}(\sU) = Q(\sO_{\sX}(\sU)).
\]

\begin{lemma}\label{lemma:localsheafofdenominators}
Let $f = g/h \in Q(\sO_{\sX}(\sX))$ where $g,h \in \sO_{\sX}(\sX)$ and $h$ is not a zero-divisor. 
Consider the presheaf $(\sO_{\sX} : f)$ which associates with each rational subset $\sU \subset \sX$, the $\sO_{\sX}(\sU)$-ideal
\[
(\sO_{\sX}(\sU) : f) := \brk{a \in \sO_{\sX}(\sU) : af\in \sO_{\sX}(\sU)}.
\]
Then, $(\sO_{\sX} : f)$ is a coherent $\sO_{\sX}$-module.
\end{lemma}

\begin{proof}
The proof is identical to that from \cite[Lemma 2.1]{bosch1982meromorphic}. 
\end{proof}

\begin{lemma}\label{lemma:localsheaf}
The presheaf $\sM_{\sX}$ defines a sheaf on $\sX$. 
\end{lemma}

\begin{proof}
The proof follows in the same manner as \cite[Lemma 2.2]{bosch1982meromorphic}, and so we will only offer a sketch. 
Let $\brk{\sU_1,\dots,\sU_n}$ denote an open cover $\sX = \Spa(A,A^{+})$ by rational subsets. 
Consider the functions $f_i \in \sM_{\sX}(\sU_i)$ such that $f_{i_{| \sU_i \cap \sU_j}} = f_{j_{| \sU_i \cap \sU_j}}$ for all $1\leq i,j\leq n$. 
By \autoref{lemma:localsheafofdenominators}, we can construct the coherent $\sO_{\sU}$-ideal $(\sO_{\sU_i} : f_i)$. 
These ideals will coincide on all intersections $\sU_i \cap \sU_j$ and hence will glue together to form a coherent $\sO_{\sX}$-module, which we denote by $\sJ$. 
Using the Noether decomposition theorem, we can see that not all elements of $\sJ(\sX)$ are zero-divisors in $\sO_{\sX}(\sX)$. 
Let $h \in \sJ(\sX)$ denote such an element. 
Now $(h_{|\sU_i})f_i \in \sO_{\sX}(\sU_i)$ for all $i$, and since they coincide on $\sU_i \cap \sU_j$, they glue to define a function $g\in \sO_{\sX}(\sX)$. 
Therefore, we can define a global element $f = g/h \in \sM_{\sX}(\sX)$ that restricts to $f_i$ on $\sU_i$ and can be shown to be unique. 
\end{proof}

Using the above results, we can globalize this construction. 
\begin{definition}
Let $\sX$ be a taut, locally strongly Noetherian adic space over $\Spa(K,K^{\circ})$. 
We define the \cdef{sheaf of meromorphic functions on $\sX$} as follows. 
Consider an affinoid open cover $\brk{\sU_i}$ of $\sX$ where each $\sU_i = \Spa(A_i,A_i^+)$ for $A_i$ a strongly Notherian ring. 
By \autoref{lemma:localsheaf}, we can define the sheaf $\sM_{\sU_i}$ for each $\sU_i$ and glue these together to get a global sheaf $\sM_{\sX}$. 
\end{definition}

\begin{remark}
The stalk at $x\in \sX$ is $Q(\sO_{\sX,x})$. Any non-zero-divisor germ $h\in \sO_{\sX,x}$ can be extended to a function on an affinoid neighborhood $\sU$ of $x$ such that $h$ is not a zero-divisor in $\sO_{\sX}(\sU)$. 
\end{remark}

\begin{definition}
For each global meromorphic function $f\in \sM_{\sX}(\sX)$, the coherent $\sO_{\sX}$-ideal $(\sO_{\sX} : f)$ can be constructed via \autoref{lemma:localsheafofdenominators}, and we will call this ideal sheaf the \cdef{ideal of denominators of $f$}. 
\end{definition}

We now define the indeterminacy locus of a global meromorphic function $f \in \sM_{\sX}(\sX)$. 
To do so, we briefly recall \cite[(1.4.1)]{huber}, which illustrates how one can define a closed adic subspace associated with a coherent $\sO_{\sX}$-ideal $\sI$. 
Let 
\[
V(\sI) := \brk{x\in \sX : \sI_{x} \neq \sO_{\sX,x}}
\]
and $\sO_{V(\sI)} := (\sO_{\sX}/\sI)_{|V(\sI)}$. 
Note that for every $x\in \sX$, the support of $v_x$ is equal to the maximal ideal $\mathfrak{m}_x$ of $\sO_{\sX,x}$. 
By definition $\sI_{y} \subset \fm_{y}$ for every $y \in V(\sI)$, and the valuation $v_y$ induces a valuation $v_{y}'$ of $\sO_{V(\sI),y}.$ 
By considering the quotient topology, one has that for every open affinoid subspace $\sV \subset \sX$, the mapping $\sO_{\sX}(\sV) \to \sO_{V(\sI)}(V(\sI) \cap \sV)$ is a quotient map. 
Moreover, the triple
\[
V(\sI) := (V(\sI),\sO_{V(\sI)}, (v_y' : y \in V(\sI)))
\]
defines an adic space, which we call the \cdef{closed adic space associated with the $\sO_{\sX}$-ideal $\sI$}. 

We now return to our goal of defining the set of poles and zeros of a global meromorphic function. 

\begin{definition}
Let $f\in \sM_{\sX}(\sX)$. 
\begin{itemize}
\item The closed adic space $P_f := V((\sO_{\sX} : f))$ associated with the coherent $\sO_{\sX}$-ideal $(\sO_{\sX} : f)$ is called \cdef{the set of poles of $f $}. 
\item Similar to the sheaf of denominators, we can define the sheaf of numerators $(\sO_{\sX} : f)\cdot f$, which is a coherent $\sO_{\sX}$-ideal. 
The closed adic space $Z_f : = V((\sO_{\sX} : f)\cdot f)$ associated with the sheaf of numerators is called \cdef{the set of zeros of $f$}. 
\item The \cdef{indeterminacy locus} of $f$ is defined as the intersection $P_f \cap Z_f$. Note that by Hilbert's nullstellensatz, $f\in \sO_{\sX}(\sX)$ if and only if $P_f = \emptyset$. 
\end{itemize}
\end{definition}

We want to show that $\sM_{\sX}(\sX)$ is a field when $\sX$ is irreducible and reduced, and to prove this result, we will need a simple lemma.

\begin{lemma}\label{lemma:unitmeromorphic}
Let $\sX$ be a reduced, taut, locally strongly Noetherian adic space over $\Spa(K,K^{\circ})$, and let $f\in \sM_{\sX}(\sX)$. 
Then $f$ is a unit in $\sM_{\sX}(\sX)$ if and only if $Z_f$ does not contain a non-empty open subspace.
\end{lemma}

\begin{proof}
We may assume that $\sX = \Spa(A,A^{+})$ is affinoid and that $A$ is complete, and so we write $f = g/h$ with $g,h\in A$ such that $h$ is not a zero-divisor in $A$. 
We have that $Z_f \subset V(g)$ where $V(g)$ denotes the vanishing locus of $g$ and both $Z_f $ and $ V(g)$ coincide on the complement of $V(h)$. 
Since $h$ is not a zero-divisor, $V(h)$ does not contain a non-empty open subspace, and moreover, we have that $Z_f$ contains an open subspace if and only if $V(g)$ contains an open subspace. 
However, this only happens if $g$ is a zero-divisor in $A$, and hence the result follows.
\end{proof}

We now record a useful corollary, which is the adic analogue of \cite[Corollary 2.5]{bosch1982meromorphic}. 

\begin{corollary}\label{coro:meromorphicfield}
If $\sX$ is an irreducible, reduced, taut, locally strongly Noetherian adic space over $\Spa(K,K^{\circ})$, then $\sM_{\sX}(\sX)$ is a field. 
\end{corollary}

\begin{proof}
This follows from \autoref{lemma:unitmeromorphic} and the fact that a closed adic subspace $\sY$ of an irreducible adic space $\sX$ which contains a non-empty open subspace must be the entire space i.e., $\sY = \sX$. 
\end{proof}

When $\sX$ is irreducible and reduced, any morphism $\sX \to \mP^{1,\an}$ that is not identically equal to $\infty$ is naturally identified with a meromorphic function $f\in \sM_{\sX}(\sX)$. We will need to know when such a morphism will separate points. 

\begin{lemma}\label{lemma:separatingpoints}
Suppose $\sX$ is an irreducible, reduced, separated, taut, locally strongly Noetherian adic space over $\Spa(K,K^{\circ})$.  
Let $x,y\in \sX$ be two distinct points. 
Then, there exists a meromorphic function $f\in \sM_{\sX}(\sX)$ such that $f(x) \neq f(y)$, where we consider $f$ as a morphism from $\sX \to \mP^{1,\an}$. 
\end{lemma}

\begin{proof}
By the valuative criterion for separatedness \cite[Proposition 1.3.7]{huber}, we have that points of $\sX$ are separated, and in particular, we have that $(\sO_{\sX,x},\sO_{\sX,x}^+)\ncong (\sO_{\sX,y},\sO_{\sX,y}^+)$. 
Moreover, we choose a germ $f\in (\sO_{\sX,x},\sO_{\sX,x}^+)$ which is not in $(\sO_{\sX,y},\sO_{\sX,y}^+)$ and then $f(x) \neq f(y)$ where we say that $f(y) = \infty$ (as $f\notin (\sO_{\sX,y},\sO_{\sX,y}^+)$). 
Since $\sX$ is irreducible and reduced, $\sO_{\sX,x}$ is a subring of $\sM_{\sX}(\sX)$ and hence $f$ defines an element of $\sM_{\sX}(\sX)$. 
\end{proof}

\section{\bf Meromorphic maps between non-Archimedean analytic spaces}
\label{sec:meromorphicmaps}

In this section, we prove a non-Archimedean variant of a result of Remmert \cite[p.~333]{remmert1957holomorphic} concerning the codimension of the indeterminacy locus of a meromorphic map between certain $K$-analytic spaces. The proofs presented here were given to us by Brian Conrad.

\begin{theorem}\label{thm:meromorphicmap}
Let $K$ be an algebraically closed, complete, non-Archimedean valued field of characteristic zero. 
Let $\sX/K$ be a normal, taut rigid analytic space, and let $\sY/K$ be a proper, reduced, taut rigid analytic space.
The indeterminacy locus of any meromorphic map $\sX\dashrightarrow \sY$ is an analytic subset of codimension at least two. 
\end{theorem}

First, we recall the definition of meromorphic mapping following Remmert. 

\begin{definition}\label{defn:meromorphic}
Let $\sX$ and $\sY$ be reduced rigid analytic spaces. A \cdef{meromorphic map} $\varphi\colon \sX \dashrightarrow \sY$ is an analytic subset $\sE\subset \sX \times \sY$ which is mapped properly onto $\sX$ by the projection map $\pr_1\colon \sX \times \sY \to \sX$ of $\sX \times \sY$ onto the first factor such that outside a nowhere dense analytic set $\sZ \subset \sX$ this map is bi-holomorphic. Moreover, $\pr_1^{-1}(\sZ)$ is nowhere dense in $\sE$, and we call the set $\sE$ the \cdef{graph of $\varphi$}. 
\end{definition}

\begin{remark}
We now explain the reason for calling $\sE$ the graph of $\varphi$. 
Let $\sX$ be a reduced rigid analytic space, and let $\sU,\sV$ be open subsets of $\sX$. 
Suppose we are given two maps $\varphi_1\colon \sU \to \sY$ and $\varphi_2\colon \sV\to\sY$ that coincide on the intersection $\sU \cap \sV$. Then the closure of the graph of $\varphi_1$ coincides with the closure of the graph of $\varphi_2$, and this graph defines  a meromorphic map $\varphi$ as in \autoref{defn:meromorphic}. 
\end{remark}

For the remainder of this section, let $\sX/K$ be a normal rigid analytic space, let $\sY/K$ be a proper, reduced rigid analytic space, let $\varphi\colon \sX \dashrightarrow \sY$ be a meromorphic map, and let $\sE$ denote the graph of $\varphi$. 
By \autoref{defn:meromorphic}, the morphism
\begin{equation}\label{eqn:morphism}
\sE \setminus \pr_1^{-1}(\sZ) \longrightarrow \sX \setminus \sZ
\end{equation}
is an isomorphism. 
We first note that there is a unique minimal $\sZ$ with the above properties and its formation is Tate-local on $\sX$. We may and do assume that $\sZ$ is minimal. 
Our goal is to show that $\sZ$ in $\sX$ has codimension at most $2$ everywhere along $\sZ$. Observe that this assertion is vacuously true when $\sZ$ is empty. Since the formation of $\sZ$ is Tate-local, we may assume that $\sX$ is affinoid, then connected, and hence irreducible, so $\sX \setminus \sZ$ is also irreducible. 
Therefore by \eqref{eqn:morphism} and the nowhere-density of $\pr_1^{-1}(\sZ)$ in $\sE$, we have that $\sE$ is irreducible. 
The image of $\pr_1\colon \sE \to \sX$ is an analytic set which contains the non-empty, Zariski open $\sX \setminus \sZ$, so it has non-empty interior in $\sX$, and since $\sX$ is irreducible, we have that this image must coincide with $\sX$ and hence the fibers of $\pr_1$ are non-empty. 

Since $\sE$ is equidimensional (by irreducibility), its pure dimension is the same as that of the irreducible affinoid $\sX$ because we may determine the dimension using $\sX \setminus \sZ$ and $\sE \setminus \pr_1^{-1}(\sZ)$. 
If there exists a section $s\colon \sX \to \sE$, then $\sX \to \sE \to \sY$ is an actual map which agrees on $\sX \setminus \sZ$ with the given meromorphic map $\varphi$, and then the minimality of $\sZ$ will imply that $\sZ$ is empty. 

We now study the fibers of the surjective morphism $\pr_1\colon \sE \to \sX$ over $\sZ$. 
Before doing so, we recall a result concerning proper morphisms.

\begin{lemma}\label{lemma:locusproper}
Let $f\colon \sX \to \sS$ be a proper map of rigid analytic spaces. Then the locus of points in $\sS$ over which the fiber of $f$ is finite is an admissible open.
\end{lemma}

\begin{proof}
We may first assume that $\sS = \Sp(A)$ is affinoid, so $\sX$ is quasi-compact and separated. 
Now consider the associated map $f^{\Ber}\colon \sX^{\Ber}\to \sS^{\Ber} = M(A)$ of Berkovich spaces. 
By \autoref{lemma:equivalenceproperfinite}.(1), the notions of properness are equivalent, and so we have that $f^{\Ber}$ is proper. 

If $s\in \sS$ has fiber $\sX_s$ that is finite, then likewise $f^{\Ber}$ has analytic fiber that is finite, hence a finite set, see  \autoref{lemma:equivalenceproperfinite}.(2).
By \cite[Corollary 3.3.11]{BerkovichSpectral}, we get an open $\sU$ in $\sS^{\Ber}$ over which $f^{\Ber}$ has finite fibers. 
In $M(A)$, a base of neighborhoods around any point is provided by rational affinoid domains, and we can arrange it to be strict since $M(A)$ is strict. 
Thus, we can find a rational affinoid $\Sp(B)$ in $\Sp(A)$ so that $M(B)$ is contained in $\sU$ and contains $s$. 
We have that $\Sp(B)$ is an admissible open in $\Sp(A)$ around $s$ over which $f$ has finite fibers (since every fiber $\sX_t$ is a quasi-compact and separated $\kappa(t)$-analytic space with $(\sX_t)^{\Ber}$ over $M(\kappa(t)$) identified with the fiber of $f^{\Ber}$ over $t$ in $M(A)$, and $\sX_t \to \Sp(\kappa(t))$ is finite if and only if the Berkovich space over $M(\kappa(t))$ is finite (in the analytic sense), see again \autoref{lemma:equivalenceproperfinite}.(2). 
\end{proof}

\begin{lemma}\label{lemma:positivedimfibers}
All fibers of the surjection $\pr_1\colon \sE \to \sX$ over $\sZ$ have positive dimension.
\end{lemma}

\begin{proof}
Suppose for the sake of contradiction that there is some $z\in \sZ \subset \sX$ which has a finite fiber. 
By \autoref{lemma:locusproper}, we have that the locus of points in $\sX$ over which is the fiber is finite is an admissible open. 
By passing to a connected affinoid neighborhood of such a $z$, we may and do assume that all of the fiber of $\pr_1$ are finite and $\sX  = \Sp(A) $ for a normal domain $A$ (since connected and normal implies irreducible and reduced). 
However, since $\pr_1$ is proper and quasi-finite, it is also finite, and hence $\sE$ is also affinoid, say $\sE = \Sp(B)$. 
Therefore, $\pr_1$ is a finite surjection $\Sp(B) \to \Sp(A)$ that restricts to an isomorphism over the complement of $\sZ$. 

Recall that $\sZ$ in $\Sp(A)$ and $\pr_1^{-1}(\sZ)$ in $\Sp(B)$ are nowhere-dense, so the irreducibillity of $\sE \setminus \pr_1^{-1}(\sZ)$ in $\sE$ forces $\sE$ to be irreducible. 
As such, we have that $B_{\red}$ is a domain and the surjective map $\Sp(B_{\red}) \to \Sp(A)$ is an isomorphism over $\Sp(A) \setminus \sZ$. We will show that the induced map on affinoid algebras $A \to B_{\red}$ is an isomorphism.

Once we have shown this, the inverse map $B_{\red}\to A$ will define a morphism $\Sp(A)\to \Sp(B_{\red}) \to \Sp(B)$ that is a section to $\pr_1\colon \Sp(B) \to \Sp(A)$ because the composition of these two maps agrees with the identity away from $\sZ$ and hence is the identity as $\sZ$ is nowhere-dense in the reduced $\Sp(A)$. 
However, we noted before \autoref{lemma:locusproper} that the existence of a section allows us to show that $\sZ$ is empty, which contradicts the existence of $z\in \sZ$. 

We now return to showing that the module finite map $A \to B_{\red}$ between Noetherian domains is an isomorphism. 
Note that this map is injective, and so by normality of $A$, the map is an isomorphism if the induced map of fraction fields has degree 1. 
Let $d$ denote the degree of the map of fraction fields. 
We know that $\Spec(B_{\red}) \to \Spec(A)$ is finite flat of degree $d$ over some non-empty Zariski open $V$, and by the Jacobson property of $A$ \cite[6.1.1/3]{BGR}, there is a closed point $t$ which is in $V$ and also away from the proper closed set corresponding to the ideal of $\sZ$. 
Moreover, $A\to B_{\red}$ induces a finite flat map of degree $d$ after completion at the maximal ideal of $t$, but that map coincides with the map upon completion at $t$ arising from the finite analytic map $\Sp(B_{\red}) \to \Sp(A)$. Recall that this latter map is an isomorphism over the complement of $\sZ$ and hence upon the completion at $t$, and therefore, we have that $d=1$. 
\end{proof}

To complete the proof, we need to show that the irreducible components of $\sZ$ have codimension at least two. 
We will need the following lemma. 

\begin{lemma}\label{lemma:finitefibers}
If $q\colon \sT \to \sS$ is a proper surjective map between $K$-analytic spaces that are equidimensional with the same dimension $d \geq 0$, then every non-empty admissible open in $\sS$ contains a point over which the fiber is finite.
\end{lemma}

\begin{proof}
The case of $d = 0$ is trivial, and we proceed by induction on $d$, so we assume that $d > 0$. 

The first step is to reduce to when $\sT$ and $\sS$ are each irreducible. We can precompose with the normalization of $\sT$ so that $\sT$ is normal with connected components $\sT_1,\dots,\sT_n$. 
Each $\sT_j$ has image that is an analytic set in an irreducible component of $\sS$.
Each irreducible component of $\sS$ must then be the image of some $\sT_j$ since $q$ is surjective, and the connected components $\sT_j$ which map onto an irreducible component of $\sS$ must factor through the normalization of that irreducible component (cf.~\cite[Theorem 2.2.4]{conrad-conn}). 

With this, we can express $q$ as a ``disjoint union" of two types of maps:~some of the $\sT_j$ map onto a connected component of the normalization of $\sS$ and some $\sT_j$ map onto a proper analytic set in an irreducible component of $\sS$. 
If we can prove our result for the first type, then by working over the Zariski-open in $\sS$ away from the images of the maps of the second type, we can conclude our result for the original $q$.
Since the normalization of $\sS$ is finite (and surjective), we can thereby reduce to the case where $\sT$ and $\sS$ are connected and normal.

Let $\sU =\Sp(A)$ be a connected affinoid in $\sS$, so it is irreducible. 
It is enough to find one fiber over $\sU$ that is finite. 
The connected components of $\sV = q^{-1}(\sU)$ are its irreducible components, and there are finitely many by the quasi-compactness of $q$. 
Note that at least one of these irreducible components maps onto $\sU$. 
By an argument similar to above with the $\sT_j$, we can find a point $u \in \sU$ that is only contained in the image of those components of $\sV$ which map onto $\sU$. 
Let $\sW$ be the union of those components. 

Pick a nonzero, non-unit $f\in A$ that vanishes at $u$. 
The pullback $f' = q^*(f)$ on $\sW$ defines an analytic set $\sW'$ mapping onto the zero locus $\sU'$ of $f$ in $\sU$. 
Endow $\sW'$ and $\sU'$ with reduced structures. 
Each of $\sW'$ and $\sU'$ is equidimensional of dimension $d-1$ since $\sU$ is irreducible and every irreducible (and hence connected) component of $\sW$ maps onto $\sU$. 
The map $q\colon \sW'\to \sU'$ satisfies the original hypotheses but with dimension $d-1$. 
Therefore, by induction, every non-empty admissible open in $\sU'$ has finite fiber in $\sW'$. 
Now, pick an open around $u$ which avoids the analytic images of the components of $\sV$ not part of $\sW$, so we get a finite fiber for $\sW'\to \sU'$ which is also a fiber for $q^{-1}(\sU')\to \sU'$, and so is a fiber of $\sV = q^{-1}(\sU) \to \sU$. 
\end{proof}

\begin{prop}\label{prop:codim}
Every irreducible component of $\sZ$ has codimension at least two. 
\end{prop}

\begin{proof}
Suppose for the sake of contradiction that $\sZ$ has an irreducible component $\sZ'$ of codimension $1$ in the connected normal affinoid $X = \Sp(A)$, in particular, we assume that $\sZ'$ has pure codimension $1$. 

First, we reduce to the case where $\sZ = \sZ'$. Pick $z'\in \sZ'$ not in any other irreducible component of $\sZ$, and let $\sW$ be a connected affinoid open in $\sX$ around $z'$ that is inside the Zariski-open complement of the finite union of the other irreducible components of $\sZ$. 
We can replace $\sX$ with such a $\sW$ (since the formation of $\sZ$ is local in $\sX$), and so we may and do assume that $\sZ$ is irreducible of pure codimension $1$. 

Next, we reduce to the case where $\sZ$ is defined by a principal ideal. 
We may assume that $\sZ$ is reduced, so $\sZ = \Sp(A/P)$ where $P$ is some prime ideal of $A$ with height $1$. 
Since $A$ is normal, we have that $A_P$ is a DVR, and so there is an affine open $\Spec(A_a)$ in $\Spec(A)$ containing $P$ for which $P_a$ is principal with generator $f \in A$. 
By looking at the map $\Sp(A) \to \Spec(A)$, we see that for the Zariski open $\sU = \brk{a \neq 0}$ in $\sX$, the intersection $\sZ \cap \sU$ is defined by the ideal generated by $f$. 
Note that $a$ is nonzero on $\Sp(A/P)$ because $a$ is not in $P$ since $\brk{P} \in \Spec(A_a)$. 
As such, we have that its sup-norm on $\Sp(A/P)$ is positive, and by replacing $a$ with some $ca^n$ for $c\in K^{\times}$ and $n$ a positive integer, we can arrange it so that this sup-norm is $1$. 
With this, we have that $\sV = \Sp(A\langle a \rangle)$ is an affinoid open in $\sX = \Sp(A)$ that meets $\sZ = \Sp(A/P)$, and $\sV \cap \sZ$ cannot equal $\sV$ since $\sZ$ has pure codimension $1$ in the irreducible $\sX$. We can replace $\sX$ with a connected component of $\sV$ that touches $\sZ$ so we retain all preceding properties and gain that the radical ideal of $\sZ$ in $\sX$ is principal, say $fA$. 

We know that $\sE$ is irreducible with the same pure dimension $d$ as the irreducible $\sX$ (so $d>0$ since $\sX$ has the irreducible subspace $\sZ$ with positive codimension), and the map $\pr_1\colon \sE \to \sX$ is surjective. 
Thus, the analytic function $f' \coloneqq \pr_1^*(f)$ on $\sE$ determines an analytic set in $\sE$ that does not exhaust the space (because its image in $\sX$ is $\sZ$ rather than $\sX$), and so it has pure dimension $d-1$. 
Moreover, the map of vanishing loci
\[
q\colon V(f')_{\red} \longrightarrow V(f) = \sZ
\]
is a proper surjection between analytic spaces each of pure dimension $d-1$. 
Now, \autoref{lemma:positivedimfibers} tells us that \textit{all} of the fibers of $q$ have position dimension, however \autoref{lemma:finitefibers} asserts that there fibers which are finite. 
Therefore, we have reached a contradiction to our original assumption that $\sZ$ has an irreducible component $\sZ'$ of codimension $1$, and hence our result follows. 
\end{proof}

\begin{proof}[Proof of \autoref{thm:meromorphicmap}]
This follows from \autoref{defn:meromorphic} and \autoref{prop:codim}. 
\end{proof}

To conclude this section, we show that the statement of \autoref{thm:meromorphicmap} carries over from rigid analytic to adic spaces. 
We note that \autoref{defn:meromorphic} works for to adic spaces \textit{mutatis mutandis}. 

\begin{prop}\label{prop:meromorphicmapadic}
Let $K$ be an algebraically closed, complete, non-Archimedean valued field of characteristic zero. 
Let $\sX$ be a normal, taut, locally of finite type adic space over $\Spa(K,K^{\circ})$, and let $\sY$ be a proper, reduced, taut, locally of finite type adic space over $\Spa(K,K^{\circ})$.
The indeterminacy locus of any meromorphic map $\sX\dashrightarrow \sY$ is an analytic subset of codimension at least two. 
\end{prop}

\begin{proof}
First, we recall that there is a functor $\mathfrak{r}$ from the category of rigid analytic spaces to the category of adic spaces, which induces an equivalence on certain subcategories (c.f.~\autoref{thm:comparisionrigidadic} and \autoref{thm:comparision}). 
We note that the image of a normal, taut (resp.~a proper, reduced, taut) rigid analytic space over $\Sp(K)$ via $\mathfrak{r}$ will be a normal, taut, locally of finite type (resp.~proper, reduced, taut, locally of finite type) adic space over $\Spa(K,K^{\circ})$ via \cite[(1.1.11) \& Remark 1.3.9.iv]{huber} and \autoref{defn:normal}. 
Next, we have that $\mathfrak{r}$ will map an open immersion of rigid analytic spaces to an open immersion of adic spaces (\textit{loc.~cit.~}(1.1.11.b)), that $\mathfrak{r}$ is fully faithful (\textit{loc.~cit.~}(1.1.11.d)), and that $\mathfrak{r}$ preserves dimensions (\textit{loc.~cit.~}(1.8.11.i)). 
The result now follows from \autoref{thm:meromorphicmap} and applying the functor $\mathfrak{r}$.
\end{proof}

\section{\bf On pseudo-$K$-analytically Brody hyperbolic varieties}
\label{sec:Brodyhyperbolic}

In this section, we offer a new definition of pseudo-$K$-analytically Brody hyperbolic which differs slightly from \cite{MorrowNonArchGGLV} and prove a result describing how one can test this notion.

To begin, we offer our new definition. 

\begin{definition} \label{defn:KBrodyMod}
Let $\sX$ be a $K$-analytic space and let $\sD\subset \sX$ be a closed subset. Then $\sX$ is \cdef{$K$-analytically Brody hyperbolic modulo $\sD$} (or:~\cdef{the pair $(\sX,\sD)$ is $K$-analytically Brody hyperbolic)} if
\begin{itemize}
\item every non-constant analytic morphism $\mathbb{G}_{m,K}^{\an} \to \sX$ factors over $\sD$, and
\item for every abelian variety $A$ over $K$ and every dense open subset $\sU\subset A^{\an}$ with $\mathrm{codim}(A^{\an}\setminus \sU)\geq 2$, every non-constant analytic morphism $\sU \to \sX$ factors over $\sD$. 
\end{itemize}
\end{definition}

\begin{definition}
A $K$-analytic space $\sX$ over $K$ is \cdef{pseudo-$K$-analytically Brody hyperbolic} if there is a proper closed subset $\sD \subsetneq \sX$ of $\sX$ such that $(\sX,\sD)$ is $K$-analytically Brody hyperbolic. 
\end{definition}

\begin{remark}\label{rmk:differentdefinitions}
\autoref{defn:KBrodyMod} differs from \cite[Definition 2.2]{MorrowNonArchGGLV} in that we require every non-constant \textit{analytic} morphism from a big \textit{analytic} open of the analytification of an abelian variety to factor over a proper closed subset, whereas \cite[Definition 2.2]{MorrowNonArchGGLV} only requires every non-constant {algebraic} morphism from a big algebraic open of an abelian variety to factor over a proper closed subset. 
In Subsection \ref{appendix}, we show that the results from \cite{MorrowNonArchGGLV} still hold with this new definition. 
\end{remark}

The goal of this section is to prove the following. 

\begin{theorem}\label{thm:equivalentBrody}
Let $\sX/K$ be an irreducible, reduced, separated $K$-analytic space and let $\sD\subset \sX$ be a closed subset. Then, $\sX$ is $K$-analytically Brody hyperbolic modulo $\sD$ if and only if for every connected, algebraic group $G/K$ and every dense open subset $\sU \subset G^{\an}$, every non-constant morphism $\sU \to \sX$ factors over $\sD$.
\end{theorem}

To prove \autoref{thm:equivalentBrody}, we follow the line of reasoning from \cite[Section 3.2]{javanpeykarXie:Finitenesspseudogroupless}, however, we need to transport several of the scheme theoretic arguments to the category of adic spaces.

\begin{lemma}\label{lemma:proxyNoether}
Let $\Spa(R,R^{\circ})$ and $\Spa(S,S^{\circ})$ be affinoid $K$-analytic spaces, and suppose $\Spa(R,R^{\circ})$ is irreducible and reduced. 
Let $\pi\colon \Spa(S,S^{\circ}) \to \Spa(R,R^{\circ})$ be a faithfully flat morphism. 
Then, there exists a dense affinoid subspace $\Spa(\tilde{R},\tilde{R}^{\circ})\subset \Spa(R,R^{\circ})$ and an affinoid subspace $\Spa(\tilde{S},\tilde{S}^{\circ})  \subset \Spa(S,S^{\circ}) $ such that the restricted map $\pi\colon\Spa(\tilde{S},\tilde{S}^{\circ}) \to \Spa(\tilde{R},\tilde{R}^{\circ})$ is finite and flat.
\end{lemma}

\begin{proof}
We will prove the result using Berkovich spaces, and then transfer it over to adic spaces. 
As such, let $M(R)$ and $M(S)$ denote the associated Berkovich affinoid spaces, and let $\pi^{\Ber}\colon M(S) \to M(R)$ denote the associated map of Berkovich affinoid spaces.

First assume that the morphism $\pi^{\Ber}\colon M(S) \to M(R)$ is quasi-finite i.e., for every $y \in M(S)$, we have that $\dim_{y}\pi^{\Ber} = 0$.
If we let $\sZ$ denote the relative interior of $\pi^{\Ber}$, then the restriction $\pi^{\Ber}_{|\sZ}$ is a finite and flat morphism, and hence open by \cite[Proposition 3.2.7]{BerkovichEtaleCohomology}. 
Using Lemma 3.1.2 of \textit{loc.~cit.,} we may find small affinoid opens $M(S') \subset M(S)$ and $M(R') \subset M(R)$ such that the morphism $\pi^{\Ber}_{|M(S')}\colon M(S') \to M(R')$ is finite and flat. 
To translate the result to affinoid $K$-analytic spaces, we use \autoref{lemma:equivalenceproperfinite}.(2) and \cite[Lemma 1.4.5.(iv) \& p.~425]{huber}.

Now suppose that $\pi^{\Ber}$ is not quasi-finite i.e., there exists a point $y \in M(S)$ such that $\dim_{y}\pi^{\Ber} = d \geq 1$. 
By \cite[Theorems 4.6 \& 3.2]{ducros_variation}, there exists an affinoid neighborhood $V = M(S')$ of $y$ such that there exists a quasi-finite morphism
\[
\varphi\colon V \to \mA_{M(R)}^d.
\]
This means that $\varphi$ is topologically proper and quasi-finite at every point of $V$. 
Let $y' \in V$, and let $x' = \varphi(y')$. Suppose that $M(R\langle T_1/r_1,\dots,T_d/r_d\rangle) \subset \mA_{M(R)}^d$ is a relative closed disk such that $x'$ lies in the relative interior of $M(R\langle T_1/r_1,\dots,T_d/r_d\rangle)$ i.e., it lies in the topological interior of $M(R\langle T_1/r_1,\dots,T_d/r_d\rangle)$. 
Let $V’ = \varphi^{-1}(M(R\langle T_1/r_1,\dots,T_d/r_d\rangle))$. We claim that  $\varphi_{|V’}$ is finite at $y’$. 

First, we note that $\varphi_{|V’}$ is quasi-finite as $M(R\langle T_1/r_1,\dots,T_d/r_d\rangle)$ is a compact, closed subset of $\mA_{M(R)}^d$, and hence it suffices to show that $\varphi_{|V'}$ is boundaryless at $y'$ i.e., $y'$ is in the relative interior $\Int(V'/M(R\langle T_1/r_1,\dots,T_d/r_d\rangle))$. 
From the above and \cite[Theorems 4.6]{ducros_variation}, we have the following morphisms
\[
\begin{tikzcd}
V' \arrow{r}{\varphi_{|V'}} \arrow[bend right = 25]{rr}{\pi^{\Ber}_{|V'}}& M(R\langle T_1/r_1,\dots,T_d/r_d\rangle) \arrow{r} &M(R). 
\end{tikzcd}
\]
By the choice of $M(R\langle T_1/r_1,\dots,T_d/r_d\rangle)$, we have that $y'$ is in the relative interior $\Int(V'/M(R))$. 
Now the claim follows from \cite[Proposition 1.5.5.(ii)]{BerkovichEtaleCohomology} as 
\[\Int(V'/M(R)) = \Int(V'/M(R\langle T_1/r_1,\dots,T_d/r_d\rangle))\, \cap \,\varphi_{|V'}^{-1}(\Int(M(R\langle T_1/r_1,\dots,T_d/r_d\rangle)/M(R))),\] and hence $y' \in \Int(V'/M(R\langle T_1/r_1,\dots,T_d/r_d\rangle))$ and therefore $\varphi_{|V’}$ is finite at $y’$. 
By Proposition 3.1.4 of \textit{loc.~cit.}, we can find affinoid neighbourhoods $V'' = M(S'')$ of $y'$ and $M(R'\langle T_1/r_1',\dots,T_d/r_d' \rangle)$ of $x'$ such that $\varphi$ induces a finite morphism 
\[
M(S'') \to M(R'\langle T_1/r_1',\dots,T_d/r_d'\rangle), 
\] 
and so we have that $S''$ is finite over $R'\langle T_1/r_1',\dots,T_d/r_d'\rangle$. 
If we consider the ideal $I = ( T_1,\dots, T_n)$ in $S''$ and take quotients, we have that $S''/I$ is finite over $R'$, and hence there exists an affinoid subset $M(S''/I)$ of $V''$ such that the morphism $M(S''/I) \to M(R')$ is finite. 
Moreover, we may assume, by Lemma 3.1.2 of \textit{loc.~cit.,} that $M(R') \subset M(R)$ is an affinoid open as we may take $M(R'\langle T_1/r_1',\dots,T_d/r_d' \rangle)$ to be arbitrarily small. 

Translating this result back to adic spaces using \autoref{lemma:equivalenceproperfinite}.(2) and \cite[Lemma 1.4.5.(iv) \& p.~425]{huber}, we have a finite morphism $\Spa(S''/I, (S^{''\circ}/I \cap S^{''\circ})^c ) \to \Spa(R',R^{'\circ})$ where $(S^{''\circ}/I \cap S^{''\circ})^c $ is the integral closure of $S^{''\circ}/I \cap S^{''\circ}$ in $S''/I$ and where $\Spa(R',R^{'\circ})$ is open in $\Spa(R,R^{\circ})$. 
Since $\Spa(R',R^{'\circ})$ is reduced, \cite[Theorem 2.21]{bhatt_hanse:sixfunctor} tells us that the flat locus is a dense open subset of $\Spa(R',R^{'\circ})$. 
Consider a smaller affinoid open of the flat locus, call it $\Spa(R'',R^{''\circ})$, and let $\Spa(S''',S^{'''\circ})$ denote its preimage in $\Spa(S''/I, (S^{''\circ}/I \cap S^{''\circ})^c )$. 
Then, we have that the restricted morphism $\pi\colon \Spa(S''',S^{'''\circ}) \to \Spa(R'',R^{''\circ})$ is finite and flat, and $\Spa(R'',R^{''\circ})$ is a dense affinoid subspace of $\Spa(R,R^{\circ})$, as desired.
\end{proof}

\begin{lemma}\label{lem:affinefactorisation}
Let $W_1 = \Spa(R,R^{\circ})$, $W_2 = \Spa(S,S^{\circ})$, and $W_3 = \Spa(T,T^{\circ})$ be affinoid $K$-analytic spaces. 
Let $\pi\colon W_2 \to W_1$ be a non-constant, faithfully flat morphism, and let $f\colon W_2\to W_3$ be a non-constant morphism. 
Assume that $W_1$ and $W_2$ are irreducible and normal, and that there exists a dense subset $E\subset W_1$ such that $f_{|\pi^{-1}(x)}$ is constant for every $x\in E$. 
Then there exists a unique $h\colon W_1\to W_3$ such that $f = h\circ \pi$.
\end{lemma}

\begin{proof}
The proof follows closely the proof of  \cite[Lemma 3.11]{javanpeykarXie:Finitenesspseudogroupless}. 
We will prove the result using rigid analytic spaces, and then transfer it over to adic spaces. 
To ease notation, we will simply identify $W_1,W_2$, and $W_3$ and the morphisms $\pi$ and $f$ with their associated images under the quasi-inverse of the functor $\mathfrak{r}$ from \cite[Proposition 4.5]{huber2}. 

We want to complete the diagram
\[
\xymatrix{
 W_1 \ar@/_/@{.>}[d]^{h}  &  W_2\ar[l]^{\pi}\ar[ld]^{f}  \\
W_3  & 
}
\]
which is equivalent to completing the diagram
\[
\xymatrix{
 R \ar[r]^{\pi^*}  &  S  \\
T. \ar[ru]^{f^*}\ar@/^/@{.>}[u]^{h^*}  & 
} 
\]
To construct $h^*$, it is enough to show that $f^*(T) \subset \pi^*(R)$. 
First, we consider $g \in f^*(T)$ and then will construct $r \in R$ such that $g=\pi^*(r)$.

By \autoref{lemma:proxyNoether}, shrinking further to ensure that the degree is constant, we may find an affinoid subspace $W_2'=\Sp(Q) \subset W_2$ and an admissible open affinoid subspace $W_1'=\Sp(R_1)  \subset W_1$ such that $\pi_{\vert_{W_2'}}\colon W_2' \rightarrow W_1' $ is finite and flat of degree $d$:
\[
\xymatrix{
R_1 \ar[r]^{\pi_{\vert_{W_2'}}^*} & Q \\
R\ar[u]\ar[r]^{\pi^*}  & S\ar[u]
}
\quad\quad
\xymatrix{
W_1' \ar@{^{(}->}[d]  & W_2' \ar[l]^{\pi_{\vert_{W_2'}}}\ar@{^{(}->}[d]\\
W_1  & W_2. \ar[l]^{\pi}
}
\]
Let $\tr(\pi^*_{\vert_{W_2'}})\colon Q \rightarrow R_1$ be the trace as defined in \cite[\href{https://stacks.math.columbia.edu/tag/0BSY}{Section 0BSY}]{stacks-project}. 
For our chosen $g \in f^*(T)$ we denote by $g_{\vert_{W_2'}}$ the restriction to $W_2'$, i.e., its image in $Q$.  We consider now $\tilde{g}:= \frac{1}{d}\tr(g_{\vert_{W_2'}}) \in R_1$; this $\tilde{g}$ will be the element $r$ such that $\pi^*(r)=g$.  

Returning to the proof, we first show that $g$ and $\pi^*(\tilde{g})$ coincide as elements in $\pi^*(R_1) \subset \pi^*(\mathrm{Frac}(R_1))=\pi^*(\mathrm{Frac}(R))$ inside $\mathrm{Frac}(S)$. 
Let $x \in W_2'(K)$ belonging to $ \pi^{-1}(W_1')$ and to $ \pi^{-1}(E) $, then we have
\[
\pi^*(\tilde{g})(x)=\tilde{g}(\pi(x)) =\frac{1}{d}\left( \sum_{y \in W_2' \vert \pi(y)=\pi(x) } g(y)\right)=g(x)
\]
where  for the last equality we used that $g \in f^*(T)$ and that $f$ is constant on $\pi^{-1}(x)$. As $\pi^{-1}(W_1')$ is dense in $W_2'$ and $\pi^{-1}(E)$ is dense in $W_2'$, then  $g=\pi^*(\tilde{g})$ as elements  of $\pi^*(\mathrm{Frac}(R))$. 

As $\pi^*\colon R \rightarrow S$ is faithfully flat, then we can use \cite[Lemma 3.10 (2)]{javanpeykarXie:Finitenesspseudogroupless} to see that 
\[
\pi^{*}(R)=S \cap \pi^*(\mathrm{Frac}(R))
\]
which gives $f^*(T) \subset \pi^*(R)$. 
Note that this gives a homomorphism $h^*\colon T\to R$ of affinoid algebras, and since a homomorphism of affinoid algebras is continuous and bounded \cite[6.1.3/1, 6.2.2/1, 6.2.3/1]{BGR}, this gives a morphism $h\colon \Sp(R) \to \Sp(T)$ of affinoid rigid analytic spaces. 
Using \cite[Proposition 4.5.(iv)]{huber2}, we arrive at our desired morphism of affinoid $K$-analytic spaces
\[
h\colon\Spa(R,R^\circ)\rightarrow \Spa(T,T^{\circ}).
\]

To conclude, we note that the construction of $\pi^*r$ is independent of $W_2'$. If we choose a different open subset $W_2''$ which is finite and flat over an open of $W_1$ and we define $\tilde{g}':= \frac{1}{d}\tr(g_{\vert_{W_2''}})$ we get that for every $x$ in $\pi^{-1}(E)$ for which both functions are defined 
\[\pi^*(\tilde{g}')(x)=\tilde{g}(\pi(x)) =\frac{1}{d}\left( \sum_{y' \in W_2'' \vert \pi(y')=\pi(x) } g(y')\right)=g(x),\]
hence $\pi^*(\tilde{g}')$ and $\pi^*(\tilde{g})$ coincide on the dense set $\pi^{-1}(E)$.
\end{proof}

\begin{lemma}\label{lemma:extendmap}
Let $\pi\colon \sY\to \sB$ be a non-constant, surjective, flat morphism between normal, irreducible, separated, locally of finite type adic spaces over $\Spa(K,K^{\circ})$. Let $f\colon \sY\to \sZ$ be a non-constant morphism of adic spaces where $\sZ$ is a $K$-analytic space. 
Let $E\subset \sB(K)$ be a dense subset. 
Assume that for every $x\in E$, the restriction $f_{|\pi^{-1}(x)}$ is constant. 
Then there exists a morphism $h\colon \sB\to \sZ$ such that $f = h\circ \pi$. 
\end{lemma}

\begin{proof}
We want to complete the diagram
\[
\xymatrix{
\sB \ar@/_/@{.>}[d]^{h}  &  \sY \ar[l]^{\pi}\ar[ld]^{f}  \\
  \sZ. & 
}
\]
We cover $\sY$ by open affinoids and as $\pi$ is surjective  it is enough to prove the theorem for any of these opens. We can further shrink these opens so that the image of each of them  in $\sZ$ is contained in an affinoid set, so we can assume that both $\sY$ and $\sZ$ are affinoid. 
Since $\sY$ is affinoid and $\sB$ is separated, the morphism $\pi\colon\sY\to \sB$ is affinoid i.e., the pre-image of an open affinoid of $\sB$ is an open affinoid in $\sY$. 
Indeed, this fact can be proved using the same argument from \cite[\href{https://stacks.math.columbia.edu/tag/01SG}{Tag 01SG}]{stacks-project}. Now if we cover $\sB$ by open affinoids $\sV_i$, then the preimages $\sW_i = \pi^{-1}(\sV_i)$ are open affinoids. 
We note that we may assume that all of the affinoids will have rings of integral elements isomorphic to the power bounded elements as our assumptions imply that the adic spaces in question come from rigid analytic spaces by \cite[Proposition 4.5.(iii)]{huber2}. 
Therefore, we are reduced to the situation of \autoref{lem:affinefactorisation}. The independence from the choices in the construction of the element $r$ in {\it loc.~cit.~}ensures that we can glue together these locally defined maps into a global map $h$. 
\end{proof}

\begin{lemma}\label{lemma:constantaffine}
Let $\sX/K$ be an irreducible, reduced, separated $K$-analytic space. Let $\sD \subset \sX$ be a closed subset such that every non-constant analytic morphism $\mG^{\an}_{m,K}$ factors over $\sD$. 
Let $G/K$ be a connected, finite type affine group scheme and let $\sU$ be a dense open of $G^{\an}$ with $\codim(G^{\an}\setminus \sU)\geq 2$. 
If $\varphi\colon \sU \to \sX$ is an analytic morphism such that $\varphi(\sU) \nsubseteq \sD$, then $\varphi$ is constant.
\end{lemma}

\begin{proof}
We proceed by induction on $\dim G$. When $\dim G \leq 1$, the result is clear because such a finite type, connected, affine group scheme contains a dense $\mG_{m,K}$. 

To show that $\varphi\colon \sU \to \sX$ is constant, it suffices to show that the meromorphic map $G^{\an} \dashrightarrow \sX$ is constant. 
Let $\sZ = \overline{\varphi(\sU)}$ denote the analytic closure of the image of $\varphi$. 
Since $G^{\an}$ is irreducible (by \autoref{lemma:absoluteproperties}.(5)) and hence $\sU$ is irreducible, the closure $\sZ \subset \sX$ is irreducible. 
Moreover, to show that $\varphi$ is constant, we may and do assume that $\sZ = \sX$. 

Since $\varphi$ is dominant, it will induce an embedding of fields of meromorphic function $\varphi^*(\sM_{\sX}(\sX)) \subseteq \sM_{G^{\an}}(G^{\an})$. 
This follows because \autoref{coro:meromorphicfield} implies that $\sM_{\sX}(\sX)$ and $\sM_{G^{\an}}(G^{\an})$ are fields (for this latter fact we use that $G^{\an}$ is smooth which follows from \cite[\href{https://stacks.math.columbia.edu/tag/047N}{Tag 047N}]{stacks-project} and \autoref{lemma:absoluteproperties}.(4)), and the dominant morphism will induce a map of fields of meromorphic functions, which must be an injection. 

We now claim that if $\varphi^*(\sM_{\sX}(\sX)) = K$, then $\varphi\colon G^{\an} \dashrightarrow \sX$ is constant. 
Indeed, if $\varphi$ was not constant, then the image of $\varphi$ contains at least two distinct values $x$ and $y$ in $\sX$. 
Since $\sX$ is separated, \autoref{lemma:separatingpoints} tells us that we can find a meromorphic function $f$ on $\sX$ that separates these points, but then the composition $f\circ \varphi$ is a non constant element of $\varphi^*(\sM_{\sX}(\sX))$, which yields a contradiction. 
Thus, it suffices to prove that $\varphi^*(\sM_{\sX}(\sX)) = K$. Below, we will suppress the subscript notation from the sheaf of meromorphic functions. 

To prove this statement, we first note that for an irreducible closed subgroup $H\subset G$, we may form the quotient
\[
\pi_H\colon G \to G/H
\]
where $G/H$ is a smooth, quasi-projective scheme by \cite[IV$_\text{A}$]{SGA3}. 
By \cite[Example 2.21]{ulirsch2017tropicalization}, the analytification functor commutes with taking (stack) quotients, and hence we may identify $(G/H)^{\an}\cong G^{\an}/H^{\an}$. 
We note that the result from \textit{loc.~cit.~}is only for Berkovich $K$-analytic spaces, however using \autoref{thm:comparision} we can transfer this identification to adic spaces. 

Since $\pi_H$ is flat and surjective, we have that the morphism
\[
\pi_{H^{\an}}\colon G^{\an} \to G^{\an}/H^{\an}
\]
is flat and surjective by \autoref{lemma:surjectiveflat}. 
Since $G^{\an}/H^{\an}$ is smooth and irreducible (as it is the image of an irreducible space), $\sM(G^{\an}/H^{\an})$ is a field by \autoref{coro:meromorphicfield}, and moreover, we have that $\pi_{H^{\an}}^*(\sM(G^{\an}/H^{\an})) \cong \sM(G^{\an})^{H^{\an}}$. 

Since $\pi_{H^{\an}}$ is partially proper, flat, and surjective, \autoref{lemma:flatopen} tells us that $\pi_{H^{\an}}(\sU)$ is a big open subset of $G^{\an}/H^{\an}$ i.e., the codimension of the complement is greater than or equal to two. 
Moreover, there exists a dense subspace $\sV\subset \pi_{H^{\an}}(\sU)$ such that for every point $x\in \sV$, the open subset $\sU \cap \pi_{H^{\an}}^{-1}(x) \subset \pi_{H^{\an}}^{-1}(x)$ is big and satisfies $\varphi(\sU \cap \pi_{H^{\an}}^{-1}(x) ) \nsubseteq \sD$. 
As an adic space, $\pi_{H^{\an}}^{-1}(x) $ is isomorphic to $H^{\an}$, and the induction hypothesis says that $\varphi_{|\pi_{H^{\an}}^{-1}(x) }$
is constant for every $x\in \sV$. 

We want to apply \autoref{lemma:extendmap} to situation where $\sY = \sU$, $\pi = \pi_{H^{\an}}$, $\sB =  \pi_{H^{\an}}(\sU)$, and $\sZ = \sX$.
Therefore, we need to verify that $\sU$ and $\pi_{H^{\an}}(\sU)$ are normal, irreducible, separated, locally of finite type adic spaces over $\Spa(K,K^{\circ})$.
As $G^{\an}$ is smooth, irreducible, separated, and locally of finite type (see \autoref{lemma:absoluteproperties}) and $\sU$ is an open subset of $G^{\an}$, we have that $\sU$ satisfies all of these properties. 
To show these properties for $\pi_{H^{\an}}(\sU)$, we first note that a quasi-projective scheme is separated \cite[\href{https://stacks.math.columbia.edu/tag/01VX}{Tag 01VX}]{stacks-project}, and hence $G^{\an}/H^{\an}$ is smooth, irreducible, separated, and locally of finite type by \autoref{lemma:absoluteproperties}. 
By \autoref{lemma:flatopen}, $\pi_{H^{\an}}(\sU)$ is an open subset of $G^{\an}/H^{\an}$, and hence has the desired properties.

Now, \autoref{lemma:extendmap} tells us that there exists a map $h_{H^{\an}}\colon \pi_{H^{\an}}(\sU) \to \sX$ such that $\varphi = h_{H^{\an}} \circ \pi_{H^{\an}}$. In particular,
\[
\varphi^{*}(\sM(\sX)) \subset \pi_{H^{\an}}^*(\sM(G^{\an}/H^{\an})) \subset \sM(G^{\an})^{H^{\an}}.
\]
By \cite[Lemma 3.14]{javanpeykarXie:Finitenesspseudogroupless}, we have that $G^{\an} = \langle H_1^{\an},\dots,H_s^{\an}\rangle$ where $H_i$ are proper connected closed subgroups of $G$. 
Therefore, we conclude that
\[
\varphi^*(\sM(\sX)) \subset \bigcap_{i=1}^s \pi_{H_i^{\an}}^*(\sM(G^{\an}/H_i^{\an})) \subset \bigcap_{i=1}^s  \sM(G^{\an})^{H_i^{\an}} = \sM(G^{\an})^{\langle H_1^{\an},\dots, H_s^{\an} \rangle} = \sM(G^{\an})^{G^{\an}} = K,
\]
as desired. 
\end{proof}

\begin{prop}\label{prop:analyticpseudogroupless}
Let $\sX/K$ be an irreducible, reduced, separated $K$-analytic space. Let $\sD \subset \sX$ be a closed subset such that every non-constant analytic morphism $\mG^{\an}_{m,K}$ factors over $\sD$ and for every abelian variety $A$ over $K$ and every dense open $\sU$ of $A^{\an}$ with $\codim(A^{\an}\setminus \sU) \geq 2$, every non-constant analytic morphism $\sU \to \sX$ factors over $\sD$. 
Then, for every connected, finite type algebraic group $G/K$ and every dense open $\sU\subset G^{\an}$ with $\codim(G^{\an}\setminus \sU)\geq 2$, every analytic morphism $\varphi\colon\sU \to \sX$ factors over $\sD$.
\end{prop}

\begin{proof}
Let $G/K$ be a connected, finite type algebraic group, and let $\sU \subset G^{\an}$ be a dense open $\sU\subset G^{\an}$ with $\codim(G^{\an}\setminus \sU)\geq 2$. 
Let $\varphi\colon\sU \to \sX$ be an analytic morphism such that $\varphi(\sU) \nsubseteq \sD$. 
We will show that $\varphi$ is constant. 

Let $H \subset G$ be the (unique) normal affine connected subgroup of $G$ such that $A \coloneqq G/H$ is an abelian variety over $K$ (see \cite[Theorem 1.1]{conradChev}). Denote by $\pi\colon G \to A$ the quotient map. 
Since $\pi$ is flat and surjective, we have that $\sV \coloneqq \pi^{\an}(\sU)$ is open in $A^{\an}$ by \autoref{lemma:flatopen} with $\codim(A^{\an}\setminus \sV) \geq 2$. 
For every $x\in A^{\an}$, denote by $G^{\an}_x \coloneqq \pi^{\an^{-1}}(x)$ and $\sU_x \coloneqq \sU \cap G^{\an}_x$. 
Since $G^{\an}\setminus \sU$ has codimension greater than or equal to two, there exists a dense open $\sV_1 \subset \sV$ such that for every $x\in \sV_1$, the closed subset $G^{\an}_x\setminus\sU_x$ has codimension greater than or equal to two. 

Let $\eta$ denote the generic point of $A^{\an}$; recall that $A^{\an}$ is irreducible by \autoref{lemma:absoluteproperties}.(5). 
If $\varphi(\sU_{\eta})$ is contained in $\sD$, as $\sU_{\eta}$ is dense in $\sU$, we have that $\varphi(\sU)$ is also contained in $\sD$, but this contradicts our original assumption. 
Thus, we have that $\varphi(\sU_{\eta	})$ is not contained in $\sD$ and hence $\sU_1 \coloneqq \varphi^{-1}(\sX \setminus \sD)$ is a non-empty open of $\sU$. 
Since $\pi^{\an}$ is flat and surjective, we have that $\sV_2 \coloneqq \pi^{\an}(\pi^{\an^{-1}}(\sV_1)\cap \sU_1)$ is open in $\sV_1$. 
Note that for every point $x\in \sV_2$, the open subset $\sU_x$ of $G^{\an}_x$ is big and $\varphi(\sU_x)$ is not contained in $\sD$. 

Therefore by our first assumption and \autoref{lemma:constantaffine}, we have that $\varphi_{|\sU_x}$ is constant for every $x\in \sV_2$. 
Since $\sV_2$ is dense in $A^{\an}$, it follows from \autoref{lemma:extendmap} that there is a morphism $h\colon \sV \to \sX$ such that $\varphi = h\circ \pi_{|\sU}^{\an}$. 
Since $\varphi(\sU) \nsubseteq \sD$, we have that $h(\sV)\nsubseteq \sD$. 
Moreover, since $\sV$ is a big open of $A^{\an}$, our second assumption implies that $h$ is constant on $\sV$, and hence $\varphi$ is constant, as desired.
\end{proof}

\begin{proof}[Proof of \autoref{thm:equivalentBrody}]
This follows from \autoref{defn:KBrodyMod} and \autoref{prop:analyticpseudogroupless}. 
\end{proof}

\section{\bf Brody special, special, arithmetically special, and geometrically special varieties}
\label{sec:specialnotions}

In this section, we recall various notions of special. 

\subsection{Special varieties in the sense of Campana}
To begin, we describe Campana's notion of special from two perspectives. The first notion concerns the non-existence of certain sheaves, while the second notion characterizes specialness in terms of the non-existence of certain fibrations. 

\subsubsection{Special varieties via Bogomolov sheaves}
First, we recall the definition of specialness in terms of Bogomolov sheaves. 
To define these sheaves, let $p$ be a positive integer and let $\sL \subset \Omega_X^p$ be a saturated, rank one, coherent sheaf. We say that $\sL$ is a \cdef{Bogomolov sheaf} for $X$ if the Iitaka dimension $\kappa(X,\sL)$ of $\sL$ equals $ p>0$.

\begin{definition}[\protect{\cite{campana2004orbifolds}}]\label{def:specialBogomolov}
A proper variety $X$ is \cdef{special} if it has no Bogomolov sheaves.
\end{definition}

\begin{example}\label{exam:fundamental}
From this definition, we have several fundamental examples of special varieties.
It is immediate that a curve $C/k$ is special if and only if the genus of $C$ is less than two and that tori and abelian varieties are special. 
Also, rationally connected varieties are special since $\Sym^m(\Omega^p) = 0$ for any $p,m>0$. 
Finally, a variety with either first Chern number equal to zero or Kodaira dimension equal to zero is special (see \cite{campana2004orbifolds} for details).
\end{example}

\subsubsection{Special varieties via fibrations of general type}
Another way to think of special varieties is via fibrations to varieties of general type. 
Consider a fibration $f\colon X\to Y$, with $X$ and $Y$ smooth and projective. 
Recall that a {fibration} $f\colon X\to Y$ is \cdef{of general type} if $K_Y(\Delta_f)$ is big and $\dim Y \geq 1$ where $\Delta_f$ is an effective $\mathbb{Q}$-Cartier divisor encoding the multiple fibers of $f$. More precisely, for every irreducible divisor $E \subset Y$, let $f^*(E)$ be the scheme-theoretic inverse image of $E$ in $X$, and write 
\[
f^*(E)=R+ \sum_i m_f(E_i) E_i,
\]
where $E_i$ are irreducible divisors in $X$, $t_i \geq 0$ are integers, and the codimension of $R$ is at least two. Define $m_f(E):=\mathrm{inf}_{i}(t_i)$. The $\mQ$-Cartier divisor $\Delta_f$ above is defined as 
\[
\Delta_f:=\sum_{E}\left( 1 -\frac{1}{m_f(E)}\right)E
\]
and is called the \cdef{multiplicity divisor of $f$}, which encodes the defect of smoothness of the fibers.

\begin{definition}\label{def:specialfibration}
A proper variety $X$ is \cdef{special} if it does not admit a fibration of general type.
\end{definition}

It is a result of Campana \cite[Theorem 2.27]{campana2004orbifolds} that these two notions are equivalent. As a consequence of \autoref{def:specialfibration}, we have that a special variety does not dominate a positive dimensional variety of general type. Using this result, combined with some non-trivial arguments, allows us to deduce the specialness of the last two families of examples of \autoref{exam:fundamental}. 

\subsubsection{Special varieties that are not proper}
Most of the aforementioned work of Campana works in much larger generality. Indeed, there is a way to define the notion of special for orbifolds, which roughly speaking is a proper variety $\overline{X}$ equipped with a orbifold divisor \[
\Delta:= \sum_{E_i}\left( 1 -\frac{1}{m_i}\right)E_i,
\]
where $E_i$ are Cartier divisors of $\overline{X}$ and $m_i \in [1,2, \ldots, +\infty]$ are almost all $1$ (so that it is a finite sum).
 
Whenever all the $m_i$ that are not $1$ are $+\infty$, we recover the classical case of logarithmic geometry, that we revise briefly for successive applications to semi-abelian schemes. Let $X/K$ be a smooth variety and consider a smooth compactification $\overline{X}$ of $X$, with normal crossing divisor complement $\Delta$; we  call $(\overline{X},\Delta)$ a logarithmic pair. We define the logarithmic canonical divisor $K_{\overline{X},\Delta}:= K_{\overline{X}}+\Delta$
and we have a notion of logarithmic general type for $X$ in terms of the logarithmic Kodaira dimension of $K_{\overline{X},\Delta}$ (see \cite[Definition 5.3]{Vojta:IntegralPointsII}). 
There is also a notion of logarithmic Bogomolov sheaf for a logarithmic pair \cite[p.~542]{campana2004orbifolds}, where one substitutes $\sL \subset \Omega_X^p$ with $\sL \subset \Omega_{\overline{X}}^p(\log \Delta)$. 
The logarithmic analogue of \autoref{def:specialBogomolov} and \autoref{def:specialfibration}, both due to Campana in the more general settings of orbifolds, are the following. 
\begin{definition}\label{def:specialBogomolovlog}
A variety $X$ is \cdef{special} if $(\overline{X},\Delta)$  does not admit any logarithmic Bogomolov sheaves.
\end{definition}
\begin{definition}\label{def:specialfibrationlog}
A variety $X$ is \cdef{special} if $(\overline{X},\Delta)$ does not admit a fibration of logarithmic general type.
\end{definition}
Campana \cite[p.~542]{campana2004orbifolds} notes that \autoref{def:specialBogomolov} and \autoref{def:specialfibration} are equivalent via the same proof as in the setting where $\Delta = \emptyset$. 
We also note that the logarithmic Kodaira dimension is a birational invariant \cite[Example 7.2.6]{campanaMH}, hence the two definitions are independent of $\overline{X}$.

\subsection{Intermezzo on weakly special varieties}
There is a weaker notion of special, called weakly special, which closely resembles \autoref{def:specialfibration}. 

\begin{definition}
We say that $X/k$ is \cdef{weakly special} if there are no finite $k$-\'etale covers $X'\to X$ admitting a dominant rational map $X'\to Z'$ to a positive dimensional variety $Z'$ of general type. 
\end{definition}

\begin{remark}
Every special variety is weakly-special \cite[Proposition 9.29]{campana2004orbifolds}. Note that weakly special curves and surfaces are special \cite[Example 7.2 (6)]{campanaMH}, but already for threefolds the converse is not true; there is an example by Bogomolov and Tschinkel (c.f.~\cite[\S 8.7]{campanaMH}).
\end{remark}

\subsection{Brody special}
Let $k = \mC$. For a complex manifold $X^{\an}/\mC$, there is a (conjectural) complex analytic characterization of specialness.

\begin{definition}\label{defn:Brodyspecial}
A complex manifold $X^{\an}$ is \cdef{Brody-special} if there is a Zariski dense holomorphic map $\mC\to X^{\an}$. 
\end{definition}

\begin{example}\label{exam:Brody}
From this definition it is immediate that for a curve $C$ over $\mC$, $C^{\an}$ is Brody-special if and only if the genus of $C$ is less than two. Also, we have that a torus $\mC^{g}$ is Brody-special, and hence an abelian variety $A/\mC$ is Brody-special via the Riemann uniformization theorem. 
Recently, Campana--Winkelmann \cite{campana2019dense} proved that the complex analytification of a rationally connected variety is Brody-special. 
\end{example}

\subsection{Geometrically special varieties}
Recently, Javanpeykar--Rousseau \cite{javanpeykar2020albanese} defined a new notion of specialness, called geometrically special, which is conjecturally equivalent to Campana's original notion.

\begin{definition}[\protect{\cite{javanpeykar2020albanese}}]\label{defn:geometricallyspecial}
A variety $X/k$ is \cdef{geometrically special over $k$} if for every dense open subset $U\subset X$, there exists a smooth, quasi-projective, connected curve $C/k$, a point in $c\in C(k)$, a point $u\in U(k)$, and a sequence of morphisms $f_i\colon C\to X$ with $f_i(c) = u$ for $i = 1,2,\dots$ such that $C\times X$ is covered by the graphs $\Gamma_{f_i} \subset C\times X$ of these maps. 
\end{definition}

\begin{example}
In \cite[Propositions 2.14 and 3.1]{javanpeykar2020albanese}, the authors prove that a rationally connected variety and an abelian variety are geometrically special over $k$, which recovers the examples from \autoref{exam:fundamental} and \autoref{exam:Brody}. 
\end{example}

\subsection{Arithmetically special varieties}
The notion of special also has a (conjectural) arithmetic counterpart, that aims to capture when the rational points of a variety are potentially dense.  

\begin{definition}\label{defn:arithmeticallyspecial}
A proper variety $X/k$ is \cdef{arithmetically special over $k$} if there exists a finitely generated subfield $k'\subset k$ and a model $\sX$ for $X$ over $k'$ such that $\sX(k')$ is dense in $\sX$. 
\end{definition}

\begin{example}
As with previous examples, a curve $C/k$ is arithmetically special if and only if the genus of $C$ is less than two and $C$ can be defined over a number field. 
Indeed, it is easy to see that a genus zero curve is arithmetically special, and it is a well-known but not obvious fact that an elliptic curve is arithmetically special. Furthermore, Faltings theorem \cite{Faltings2} asserts that a curve of genus $g\geq 2$ is not arithmetically special. 
By \cite[Section 3]{HassetTschinkel:AbelianFibrations}, any abelian variety is arithmetically special. 
For further examples of arithmetically special varieties, we refer the reader to \cite{HarrisTsch, bogomolov2000density, laiNaka:uniformpotentialdensity}. 
\end{example}

\begin{remark}
All of the above notions of special are stable under birational morphisms, finite \'etale covers, and products. We refer the reader to \cite[Section 2]{javanpeykar2020albanese} for details. 
\end{remark}

\section{\bf A non-Archimedean analytic characterization of special and weakly special}
\label{sec:nonArchspecial}

In this section, we offer our definition of special for a $K$-analytic space and describe some basic properties of these special $K$-analytic spaces.

\begin{definition}\label{defn:Kanspecial}
We say that $\sX$ is \cdef{$K$-analytically special} if there exists a connected, finite type algebraic group $G/K$, a dense open subset $\sU\subset G^{\an}$ with $\codim(G^{\an}\setminus \sU) \geq 2$, and an analytic morphism $\sU \to \sX$ which is Zariski dense. 
\end{definition}

\begin{example}
From \autoref{defn:Kanspecial}, we can immediately find several examples of $K$-analytically special varieties. First, a curve $C/K$ is $K$-analytically special if and only if the the genus of $C$ is less than two, and second, a connected, finite type algebraic group is $K$-analytically special. 
\end{example}

While checking $K$-analytic specialness with big {analytic} opens of algebraic groups may seem unnatural, we do so in order to incorporate the following types of examples. 

\begin{example}\label{exam:whybig}
Let $A/K$ be a simple abelian surface with good reduction. From \autoref{defn:Kanspecial}, it is clear that $A\setminus\brk{0}$ is $K$-analytically special, but if we did not test specialness on big analytic opens of algebraic groups then $A \setminus\brk{0}$ would not be $K$-analytically special. 
To see this, we note that if $G/K$ is an algebraic group and $G^{\an} \to (A\setminus\brk{0})^{\an}$ a non-constant, analytic morphism, then composition $G^{\an} \to (A\setminus\brk{0})^{\an} \to A^{\an}$ will be the translate of a group homomorphism, and hence a point by our assumption that $A$ is simple.

Indeed, using Chevalley's decomposition theorem \cite[Theorem 1.1]{conradChev}, we may write $G$ as an extension of an abelian variety $B$ by a normal, affine group $H$. Note that each pair of points in $H$ will be connected by a $\mG_m$, and hence the image of $H^{\an} $ in $ A^{\an}$ will be a point due to work of Cherry \cite[Theorem 3.2]{Cherry}. Moreover, we have that the morphism $G^{\an} \to A^{\an}$ will factor through the analytification of the abelian variety $B^{\an}$, but since $B$ is proper, rigid analytic GAGA \cite{kopfGAGA} implies that the morphism $B^{\an}\to A^{\an}$ is algebraic i.e., it is the analytification of an algebraic morphism $B\to A$. It is well-known that any morphism $B\to A$ is the translate of a group homomorphism, and therefore the image of $G^{\an} \to A^{\an}$ will be the translate of an abelian subvariety. 
Since $A$ was assumed to be simple, the image of $G^{\an}$ in $A^{\an}$ will be a point, and moreover, $(A\setminus\brk{0})^{\an}$ would not be $K$-analytically special. 
 
We remark that the idea of testing specialness and hyperbolicity on big open of algebraic groups goes back to Lang  \cite{Lang} and was studied by Vojta in \cite{VojtaLangExc}. In this latter work, Vojta showed that for $A$ an abelian variety over $\mC$ and $U$ a dense open subset of $A$ with $\codim(A\setminus U)\geq 2$, $U^{\an}$ is Brody-special by \textit{loc.~cit.~}Section 4. 
\end{example} 

\subsection{Basic properties:~birational invariance, products, and ascending along finite \'etale covers}
For the remainder of this section, we prove that our notion of $K$-analytically special is preserved under birational morphisms, products, and finite \'etale covers. 

\begin{lemma}\label{lemma:birationalinvariance}
Let $\sX \dashrightarrow \sY$ be a bi-meromorphic morphism between proper, integral $K$- analytic spaces. Then, $\sX$ is $K$-analytically special if and only if $\sY$ is $K$-analytically special. 
\end{lemma}

\begin{proof}
Suppose that $\sX$ is $K$-analytically special so there exists a connected, finite type algebraic group $G/K$, an open dense subset $\sU \subset G^{\an}$ with $\codim(G^{\an}\setminus \sU) \geq 2$, and a Zariski dense, analytic morphism $\varphi\colon\sU \to \sX$. 
Since $\varphi$ is Zariski dense, the composition $\sU \to \sX \dashrightarrow \sY$ defines a meromorphic map from $G^{\an} \dashrightarrow \sY$. 
By  \cite[\href{https://stacks.math.columbia.edu/tag/047N}{Tag 047N}]{stacks-project} and \autoref{lemma:absoluteproperties}.(4), $G^{\an}$ is smooth, and so by \autoref{prop:meromorphicmapadic}, this meromorphic map is defined on some dense open $\sU'$ of $G^{\an}$ with $\codim(G^{\an}\setminus \sU') \geq 2$. Therefore, $\sY$ is $K$-analytically special. To conclude, we note that the proof of the converse statement follows in the exact same manner. 
\end{proof}

\begin{remark}\label{remark:whybiganalytic}
We remark that if we tested $K$-analytic specialness on big \textit{algebraic} opens of connected, finite type algebraic groups, then bi-meromorphic invariance would not follow. Indeed, the proof of \autoref{lemma:birationalinvariance} boils down to showing that a meromorphic map between $K$-analytic spaces is defined on an open subset whose complement has codimension at least $2$. However, we do not know of a way to guarantee that this open subset is algebraic, and therefore, it is crucial that we test $K$-analytic specialness on big \textit{analytic} opens of connected, finite type algebraic groups. 
\end{remark}

\begin{lemma}
Let $\sX/K$ and $\sY/K$ be $K$-analytic spaces which are both $K$-analytically special. Then the product $\sX\times \sY$ is $K$-analytically special. 
\end{lemma}

\begin{proof}
Let $G$ (resp.~$G'$) be a connected, finite type algebraic group over $K$, let $\sU \subset G^{\an}$ (resp.~$\sU'\subset G^{'\an}$) be a dense open subset with $\codim(G^{\an}\setminus \sU) \geq 2$ (resp.~$\codim(G^{'\an}\setminus \sU') \geq 2$) such that there exists a Zariski dense analytic morphism $\varphi\colon \sU \to \sX$ (resp.~$\varphi'\colon \sU'\to \sY$). 
We have that $G \times G'$ is a connected, finite type algebraic group over $K$ and that the fibered products $G^{\an} \times G^{'\an}$ and $\sU \times \sU'$ exist in the category of locally of finite type adic spaces over $\Spa(K,K^{\circ})$ by \cite[1.2.2.(a)]{huber}. 
The same argument from \cite[\href{https://stacks.math.columbia.edu/tag/01JR}{Tag 01JR}]{stacks-project} tells us that $\sU \times \sU'$ is a dense open subset of $G^{\an} \times G^{'\an}$. Moreover, we claim that $\codim((G^{\an} \times G^{'\an}) \setminus (\sU \times \sU'))\geq 2$.
Indeed, using \cite[Theorem 5.1.3.1]{conrad-conn} and \cite[1.8.11.(i)]{huber}, we deduce that  $\dim(G^{\an} \times G^{'\an}) = \dim(G) + \dim(G')$ and the claim now follows from \cite[1.8.8]{huber}. 
To conclude, we have that the morphism $(\varphi \times \varphi')\colon \sU \times \sU' \to \sX \times \sY$ is Zariski dense, and so $\sX\times \sY$ is $K$-analytically special. 
\end{proof}

\begin{lemma}\label{lemma:ascendfiniteetale}
Let $\sX\to \sY$ be a finite \'etale morphism between proper $K$-analytic spaces. Then $\sX$ is $K$-analytically special if and only if $\sY$ is $K$-analytically special.
\end{lemma}

\begin{proof}
If $\sX$ is $K$-analytically special, then it follows that $\sY$ is $K$-analytically special since $\sX\to \sY$ is finite \'etale.

Now suppose that $\sY$ is $K$-analytically special. 
Let $G/K$ be a connected, finite type algebraic group, let $\sU \subset G^{\an}$ open dense such that $\codim(G^{\an}\setminus \sU) \geq 2$, and let $\sU \to \sY$ be a Zariski dense, analytic morphism. 
By  \cite[1.2.2.(a)]{huber}, we can construct the fibered product diagram
\[
\begin{tikzcd}
\sV \arrow{r}\arrow{d} & \sX\arrow{d} \\
\sU\arrow{r} & \sY,
\end{tikzcd}
\]
and since finite \'etale covers are stable under base change \cite[1.4.5.(i) \& 1.6.7.(iv)]{huber}, we have that $\sV \to \sU$ is finite \'etale. 
By \cite[\href{https://stacks.math.columbia.edu/tag/047N}{Tag 047N}]{stacks-project} and \autoref{lemma:absoluteproperties}.(4), $G^{\an}$ is smooth over $K$, and so purity of the branch locus (see e.g., \cite[Chapter 3, Theorem 2.1.11]{andre-per} or \cite[Corollary 2.15]{Hanse:Vanishing}) implies that the finite \'etale morphism $\sV\to \sU$ extends to a finite \'etale morphism $\sG'\to G^{\an}$. 
By the non-Archimedean analogue of Riemann's existence theorem \cite[Theorem 3.1]{Lutkebohmert:Riemannexistence}, the finite \'etale morphism $\sG'\to G^{\an}$ algebraizes, in particular there is a finite \'etale morphism of locally of finite type schemes $G'\to G$ whose analytification coincides with $\sG'\to G^{\an}$. 
Note that every connected component $G''$ of $G'$ has the structure of a connected, finite type group scheme over $K$, and with this structure the morphism $G''\to G$ is a homomorphism. 
Since smooth morphisms preserve codimension, we have that $\codim(G^{''\an}\setminus \sV) \geq 2$, and note that $\sV \to \sU \to \sY$ (and hence $\sV \to \sX \to \sY$) is Zariski dense.

To conclude, we need to show that the image of $\sV \to \sX$ is Zariski dense. 
Suppose to the contrary. 
Let $\sZ$ denote the complement of the Zariski closure of the image of $\sV$ inside of $\sX$. 
By assumption, $\sZ$ is a non-empty open subset of $\sX$. 
Since $\sX\to \sY$ is \'etale, the image of $\sZ$ in $\sY$ is open \cite[1.7.8]{huber}, and so we have an open subset of $\sY$ which is not in the image of $\sV$. 
However, this contradicts the fact that $\sV \to \sX \to \sY$ is Zariski dense, and therefore, we have that $\sV \to \sX $ is Zariski dense and so $\sX$ is $K$-analytically special.
\end{proof}

Recall that one can characterize special varieties as those which do not admit a fibration of general type. 
In the non-Archimedean setting, \autoref{thm:equivalentBrody} allows us to show that a $K$-analytically special variety cannot dominant a pseudo-$K$-analytically Brody hyperbolic variety.

\begin{theorem}[= \autoref{thm:nofibrationtoBrody}]
Let $\sX$ and $\sY$ be irreducible, reduced, separated $K$-analytic spaces over $K$. 
If $\sY$ is $K$-analytically special and $\sX$ is a positive dimensional pseudo-$K$-analytically Brody hyperbolic variety, then there is no dominant morphism $\sY \to \sX$.
\end{theorem}

\begin{proof}
Suppose there existed a dominant morphism from $\sY \to \sX$. 
Then there exists a connected, finite type algebraic group $G$, a dense open subset $\sU$ of $G^{\an}$ with $\codim(G^{\an}\setminus \sU)\geq 2$, and an analytic morphism $\sU\to \sY\to \sX$, which is Zariski dense since the composition of dominant morphisms is dominant. 
However, this contradicts \autoref{thm:equivalentBrody}, and therefore, we have that there cannot exist a dominant morphism from $\sY\to \sX$.
\end{proof}

To conclude this section, we define the notion of $K$-analytically weakly special. 

\begin{definition}
We say that $\sX$ is \cdef{$K$-analytically weakly special} if there are no finite $K$-\'etale covers $\sX'\to \sX$ admitting a dominant meromorphic map $\sX'\to \sZ'$ to a positive dimensional $K$-analytic space $\sZ'$ which is pseudo-$K$-analytically Brody hyperbolic. 
\end{definition}

With this definition, we have the following corollary to \autoref{lemma:ascendfiniteetale} and \autoref{thm:nofibrationtoBrody}.

\begin{corollary}
A $K$-analytically special $K$-analytic space is $K$-analytically weakly special. 
\end{corollary} 

\section{\bf $K$-analytically special subvarieties of semi-abelian varieties and properties of quasi-Albanese maps}
\label{sec:specialsubvarietiessemiabelian}

In this section, we will prove \autoref{thm:closedabelianspecial} and \autoref{coro:specialabelian} and use these results to deduce some properties of the quasi-Albanese map of a  $K$-analytically special variety.

To begin, we recall the statement of \autoref{thm:closedabelianspecial}. 

\begin{theorem}[= \autoref{thm:closedabelianspecial}]
Let $X$ be a closed subvariety of a semi-abelian variety $G$ over $K$. Then $X$ is the translate of a semi-abelian subvariety if and only if $X^{\an}$ is $K$-analytically special.
\end{theorem}

To prove Theorem \ref{thm:closedabelianspecial}, we will use some deep theorems about closed subvarieties of semi-abelian varieties. 
First, we recall progress on the Green--Griffiths--Lang--Vojta conjecture for closed subvarieties of semi-abelian varieties.

\begin{theorem}[\protect{\cite{Abram, Nogu, MorrowNonArchGGLV}}] \label{thm:starting_point}
Let $X$ be a closed subvariety of a semi-abelian variety $G$ over $K$. Then, $X$ is of logarithmic general type if and only if $X^{\an}$ is pseudo-$K$-analytically Brody hyperbolic. 
\end{theorem}
  
We note that \autoref{thm:starting_point} was proved using the definition of pseudo-$K$-analytically Brody hyperbolic from \cite{MorrowNonArchGGLV}, and we will need to know that this result holds with our new definition (\autoref{defn:KBrodyMod}).
The proof of this fact will occupy the next subsection. 

\subsection{\bf Extending meromorphic maps to analytifications of semi-abelian varieties}\label{appendix}
The goal of this subsection is to prove the following extension result concerning meromorphic maps from smooth, irreducible analytic space to analytifications of semi-abelian varieties. 

\begin{theorem}\label{thm:extendanalyticmorphism}
Let $\sZ$ be a smooth, irreducible $K$-analytic space, and let $\sU \subset \sZ$ be a dense open with $\codim(\sZ\setminus \sU) \geq 2$. 
Let $G/K$ be a semi-abelian variety. 
Then any analytic morphism $\sU \to G^{\an}$ uniquely extends to an analytic morphism $\sZ \to G^{\an}$. 
\end{theorem}

The algebraic variant of \autoref{thm:extendanalyticmorphism} is well-known (see e.g., \cite[Theorem 4.4.1]{BLR} and \cite[Lemma A.2]{MochizukiAbsoluteAnabelian}).
To prove our results, we begin by studying the case when $G$ is proper (i.e., when $G$ is an abelian variety). We first show that the analytic Picard group of $\sU$ is in bijection to the analytic Picard group of $\sZ$ (\autoref{lemma:Picard}), and then the result follows using similar reasoning as \cite[Corollary 8.4.6]{BLR}. 
Next, we prove the claim in the setting where $G$ is a (split) torus using the non-Archimedean variant of the second Hebbarkeitssatz. 
Finally, we conclude the proof using these two pieces and a descent argument.

To start, we prove our result on analytic Picard groups. 

\begin{lemma}\label{lemma:Picard}
Let $\sZ$ be a smooth, irreducible $K$-analytic space, and let $\sU \subset \sZ$ be a dense open with $\codim(\sZ\setminus \sU) \geq 2$. Then, $\Pic(\sU) $ is in bijection with $ \Pic(\sZ)$. 
\end{lemma}

\begin{proof}
By the non-Archimedean version of the Remmert--Stein theorem \cite{Lutkebohmert:RemmertStein}, we have that the group of Weil divisors on $\sU$ is isomorphic to the group of Weil divisors on $\sZ$. 
Since $\sZ$ (resp.~$\sU$) is smooth, the group of Weil divisors on $\sZ$ (resp.~$\sU$) isomorphic to the group of Cartier divisors on $\sZ$ (resp.~$\sU$) by \cite[Theorem 8.9]{mitsui2011bimeromorphic}. 
Now since $\sZ$ (resp.~$\sU$) is irreducible, the classical argument (see e.g., \cite[Chapter II, Proposition 6.15]{Har77}) tells us that Cartier divisors on $\sZ$ (resp.~$\sU$) are in bijection with line bundles on $\sZ$ (resp.~$\sU$). 
Therefore, we have that line bundles on $\sU$ are in bijection with line bundles on $\sZ$. 
\end{proof}

\begin{remark}
The algebraic variant of \autoref{lemma:Picard} is well-known, and the result is actually not true for complex analytic manifolds (see \cite[Section 2.3]{huybrechts2005complex} for a discussion). Indeed, in the complex analytic setting, it is not true that Cartier divisors on a smooth complex manifold $X(\mC)$ are in bijection with line bundles on $X(\mC)$; instead, they are in bijection with line bundles $\sL$ on $X(\mC)$ such that $H^0(X(\mC),\sL)\neq 0$. 
Moreover, there are many examples of line bundles that do not come from Cartier divisors on complex manifolds and line bundles on big, dense opens which do not extend to the entire space (see e.g., \cite[Remark 2.3.21]{huybrechts2005complex}).

The issue in the complex analytic setting is that when $X/\mC$ is separated, $X(\mC)$ is Hausdorff, and hence if $X$ has positive dimension, it cannot also be irreducible in the complex analytic topology. 
The lack of irreducibility prevents the scheme-theoretic argument from going through as we need to know that if the restriction of a sheaf to each open covering is constant, then sheaf is constant. 
Therefore, it is essential that we work with adic spaces and not Berkovich spaces (c.f.~\autoref{remark:nonHausdorff}). 
\end{remark}

\begin{lemma}\label{lemma:extendabelianvariety}
Let $\sZ$ be a smooth, irreducible $K$-analytic space, and let $\sU \subset \sZ$ be a dense open with $\codim(\sZ\setminus \sU) \geq 2$. Let $A/K$ be an abelian variety.  
Then any analytic morphism $\varphi\colon \sU \to A^{\an}$ uniquely extends to an analytic morphism $\wt{\varphi}\colon \sZ \to A^{\an}$. 
\end{lemma}

\begin{proof}
We will closely follow the proof from \cite[Corollary 8.4.6]{BLR}. 
First, we base change $A/K$ to an abelian variety $A_\sZ \coloneqq A\times_K \sZ $. 
Recall that $A \cong A^{**}$ \cite[Theorem 8.4.5]{BLR}, and so by rigid analytic GAGA \cite{kopfGAGA}, $A^{\an}\cong A^{**\an}$, where $A^{*\an} \cong A^{*}$ represents the adic Picard functor of $A^{\an}$ (c.f.~\cite[Section 1]{BLII}). 
Now a morphism $\varphi\colon \sU \to A^{\an}_\sZ$ corresponds to an analytic line bundle on $A^{*\an} \times_{\sZ} \sU$.
Since $\sZ/K$ is smooth, we have that $A^{*\an} \to \sZ$ is smooth, and so \autoref{lemma:Picard} implies that an analytic line bundle on 
$A^{*\an} \times_{\sZ} \sU$ uniquely extends to a line bundle on $A^{*\an}_{\sZ}$, and hence gives rise to an extension $\wt{\varphi}\colon \sZ \to A^{**\an}_{\sZ}$. 
\end{proof}

\begin{lemma}\label{lemma:extendtorus}
Let $\sZ/K$ be a smooth $K$-analytic space, and let $\sU \subset \sZ$ be an open subset with $\codim(\sZ\setminus \sU) \geq 2$. 
For $T/K$ a split torus, any analytic morphism $\sU \to T^{\an}$ uniquely extends to an analytic morphism $\sZ\to T^{\an}$. 
\end{lemma}

\begin{proof}
Note that it suffices to show that if $\sL$ is any line bundle on $\sZ$ that admits a generating section $s_{\sU} \in H^0(\sU,\sL)$, then $s_{\sU}$ extends to a generating section of $\sL$ over $\sZ$. 
This follows from the non-Archimedean variant of the second Hebbarkeitssatz \cite{Lutkebohmert:RemmertStein}. 
\end{proof}

\begin{proof}[Proof of \autoref{thm:extendanalyticmorphism}]
Let $G/K$ denote a semi-abelian variety and suppose we have the following presentation
\[
0 \to T \to G \to A \to 0
\]
where $T$ is a split torus over $K$ and $A$ is an abelian variety over $K$. 
First we note that the quotient map $G\to A$ is faithfully flat with smooth fibers and hence smooth by \cite[Tag 01V8]{stacks-project}. 
In fact, we have that $G\times_A G \cong T \times G$ via $(g,h) \to (g-h,g)$ and similarly $G\times_A G \times_A G \cong T\times T\times G$. 
We note that since analytification commutes with fiber products, we have that 
$G^{\an} \times_{A^{\an}} G^{\an} \cong T^{\an}\times G^{\an}$ and 
$G^{\an}\times_{A^{\an}} G^{\an} \times_{A^{\an}} G^{\an} \cong T^{\an}\times T^{\an}\times G^{\an}$.

Let $\sZ$ be a smooth, irreducible $K$-analytic space, and let $\sU \subset \sZ$ be a dense open with $\codim(\sZ\setminus \sU) \geq 2$. 
By \autoref{lemma:extendabelianvariety}, we can uniquely extend the composition $\sU\to G^{\an} \to A^{\an}$ to $\sZ \to A^{\an}$, which gives use the commutative diagram
\[
\begin{tikzcd}
\sU \arrow[right hook->]{r} \arrow{d} & \sZ \arrow{d} \\
G^{\an} \arrow{r} & A^{\an}.
\end{tikzcd}
\]
Pulling back along the map $G^{\an} \to A^{\an}$, we have the commutative diagram
\[
\begin{tikzcd}
\sU \times_{A^{\an}} G^{\an} \arrow[right hook->]{r} \arrow{d} & \sZ \times_{A^{\an}} G^{\an} \arrow{d} \\
G^{\an} \times_{A^{\an}} G^{\an} \arrow{r} & G^{\an}.
\end{tikzcd}
\]
Recalling that $G^{\an} \times_{A^{\an}} G^{\an} \cong T^{\an}\times G^{\an}$, we get a map $\sU \times_{A^{\an}} G^{\an} \to T^{\an}\times G^{\an} \to T^{\an}$ via the first projection. By \autoref{lemma:extendtorus}, this uniquely extends to a map $\sZ \times_{A^{\an}} G^{\an} \to T^{\an}$ since $\sU\times_{A^{\an}} G^{\an} \subset \sZ\times_{A^{\an}}G^{\an}$ is an open immersion of smooth $K$-analytic spaces. This gives us the following diagonal arrow in the above diagram
\[
\begin{tikzcd}
\sU \times_{A^{\an}} G^{\an} \arrow[right hook->]{r} \arrow{d} & \sZ \times_{A^{\an}} \arrow{ld} G^{\an} \arrow{d} \\
T^{\an} \times G^{\an} \arrow{r} & G^{\an}.
\end{tikzcd}
\]
Note that the top triangle commutes since it does so after composing with the first and second projections $T^{\an}\times G^{\an} \to T^{\an}$ (by the extension property) and  $T^{\an}\times G^{\an} \to G^{\an}$ (which was given). Moreover, the bottom triangle commutes because everything is a morphism over $G^{\an}$. 

We need to show that the morphism $\sZ \times_{A^{\an}} G^{\an} \to G^{\an} \times_{A^{\an}} G^{\an}$ descends to $\sZ \to G^{\an}$. 
To do this, we use the theory of faithfully flat descent in the category of rigid analytic varieties, which was developed in \cite[Section 4.2]{conrad2006relative}. 
First, we note that the morphism of $K$-analytic spaces $G^{\an} \to A^{\an}$ admits local {fpqc }quasi-sections via \textit{loc.~cit.~}Theorem 4.2.2 as it is the analytification of the faithfully flat morphism of $K$-schemes $G\to A$. Moreover, we may appeal to \textit{loc.~cit.~}Theorem 4.2.3. We note that while these results are for rigid analytic spaces, they carry over for the $K$-analytic spaces we consider due to \autoref{thm:comparisionrigidadic} and \autoref{thm:comparision}. 

In order to utilize \textit{loc.~cit.~}Theorem 4.2.3, we need to check the cocycle condition that the two pullbacks along $(\cdot) \times_{A^{\an}} G^{\an} \times_{A^{\an}} G^{\an} \substack{\rightarrow\\[-1em] \rightarrow} (\cdot) \times_{A^{\an}} G^{\an}$ agree. 
These pullbacks agree above the dense open $\sU \to \sZ$ and $G^{\an}\times_{A^{\an}} G^{\an} \times_{A^{\an}} G^{\an} \cong T^{\an}\times T^{\an}\times G^{\an}$ over $G^{\an}$ (via projection onto the last component), so the uniqueness statement for morphisms to $T^{\an} \times T^{\an}$ shows that they agree everywhere. Therefore, we have that $\sZ \times_{A^{\an}} G^{\an} \to G^{\an} \times_{A^{\an}} G^{\an}$ descends to $\sZ \to G^{\an}$, and by construction, the restriction to $\sU$ is the initial morphism $\sU \to G^{\an}$. 
\end{proof}

We now use \autoref{thm:extendanalyticmorphism} to deduce that an analytic map from a big open subset of an abelian variety to a semi-abelian varietiy uniquely extend to an algebraic morphism between the abelian variety and the semi-abelian variety. 

\begin{prop}\label{prop:extendalgebraic}
Let $A/K$ be an abelian variety, and let $\sU \subset A^{\an}$ be a dense open with $\codim(A^{\an}\setminus \sU) \geq 2$. 
Let $G/K$ be a semi-abelian variety.  
Then any analytic morphism $\varphi\colon \sU \to G^{\an}$ uniquely extends to an algebraic morphism $\wt{\varphi}\colon A\to G$.
\end{prop}

\begin{proof}
By \autoref{thm:extendanalyticmorphism}, we have that the analytic morphism $\varphi\colon \sU \to G^{\an}$ extends to an analytic map $\wt{\varphi}\colon A^{\an} \to G^{\an}$. 
The result now follows from \cite[Lemma 2.15]{JVez}. 
\end{proof}

To conclude this subsection, we show that the results from \cite{MorrowNonArchGGLV} remain valid with our new definition of $K$-analytically Brody hyperbolic modulo $\sD$ (\autoref{defn:KBrodyMod}).

\begin{prop}\label{prop:Brodyequivalence}
Let $X/K$ be a quasi-projective variety which is a closed subvariety of a semi-abelian variety, and let $\Delta \subset X$ be a closed subset. 
Then, $X^{\an}$ is $K$-analytically Brody hyperbolic modulo $\Delta^{\an}$ if and only if every analytic morphism $\mG_{m,K}^{\an} \to X^{\an}$ factors over $\Delta^{\an}$ and for every abelian variety $A/K$ and every dense open subset $U\subset A$ with $\codim(A\setminus U)\geq 2$, every morphism $U\to X$ factors over $\Delta$. 
\end{prop}

\begin{proof}
It suffices to show the second implication, namely that if every analytic morphism $\mG_{m,K}^{\an} \to X^{\an}$ factors over $\Delta^{\an}$ and for every abelian variety $A/K$ and every dense open subset $U\subset A$ with $\codim(A\setminus U)\geq 2$, every morphism $U\to X$ factors over $\Delta$, then $X^{\an}$ is $K$-analytically Brody hyperbolic modulo $\Delta^{\an}$. 
By \autoref{thm:equivalentBrody}, it suffices to show that for every abelian variety $A/K$ and every dense open subset $\sU\subset A^{\an}$ with $\codim(A^{\an}\setminus \sU)\geq 2$, every morphism $\sU\to X^{\an}$ factors over $\Delta^{\an}$. 
\autoref{prop:extendalgebraic} tells us that the morphism $\sU\to X^{\an}$ uniquely extends to an algebraic morphism of $K$-schemes $A\to X$. 
The desired result follows because \cite[Lemma A.2]{MochizukiAbsoluteAnabelian} asserts that the morphism $U\to X$ uniquely extends to a morphism $A\to X$, and the image of the map $A\to X$ must factor through $\Delta$. 
\end{proof}

\begin{remark}\label{rem:Brodymoduloempty}
We can use \autoref{thm:extendanalyticmorphism} and \autoref{prop:Brodyequivalence} to deduce that for $X/K$ a quasi-projective variety which is a closed subvariety of a semi-abelian variety, $X^{\an}$ is $K$-analytically Brody hyperbolic modulo $\emptyset$ if and only if $X^{\an}$ is $K$-analytically Brody hyperbolic, in the sense of \cite[Definition 2.3]{JVez}. Indeed, this follows immediately from the argument from \cite[Remark 2.6]{MorrowNonArchGGLV} and the previously mentioned results. 
\end{remark}

%
%

\subsection{Proof of \autoref{thm:closedabelianspecial}}
We now return to the proof of \autoref{thm:closedabelianspecial} and begin by recalling the construction of the Ueno fibration for closed subvarieties of semi-abelian varieties. 

\begin{definition}\label{defn:stablizier}
Let $X$ be a closed subvariety of a semi-abelian variety $G$. 
\begin{enumerate}
\item We define the \cdef{stabilizer of $X$ in $G$} as the maximal closed subgroup $\Stab(X,G)$ of $G$ such that $\Stab(X,G) + X = X$. 
\item We denote the identity component of the closed subgroup $\Stab(X,G)$ by $B(X,G)$. 
\end{enumerate}
\end{definition}

\begin{lemma}\label{lemmastabilizersemiabelian}
Let $X$ be a closed subvariety of a semi-abelian variety $G$. 
The closed subgroup $B(X,G)$ is a semi-abelian subvariety of $G$.
\end{lemma}

\begin{proof}
By definition, $B(X,G)$ is connected, and since $B(X,G)$ is a closed subgroup of $G$, we have that $B(X,G)$ is an algebraic group. Moreover, we have that $B(X,G)$ is a connected and smooth subgroup of $G$ (see \cite[\href{https://stacks.math.columbia.edu/tag/047N}{Tag 047N}]{stacks-project} for the statement on smoothness). 
The result now follows from \cite[Corollary 5.4.6.(1)]{brion2017some}. 
\end{proof}

\begin{definition}[\protect{\cite[Definition 1.2]{Vojta:IntegralPointsII}}]\label{defn:Ueno}
Let $X$ be a closed subvariety of a semi-abelian variety $G$. Consider the quotient $G/B(X,G)$, which is a semi-abelian variety by \cite[Corollary 5.4.6.(1)]{brion2017some}. The restriction to $X$ of the quotient map $G\to G/B(X,G)$ exhibits $X$ as a fiber bundle with fiber $B(X,G)$. This map $X \to X/B(X,G)$ is called the \cdef{Ueno fibration} of $X$.

To summarize, we have the following diagram
\[
\begin{tikzcd}
G \arrow[two heads]{r} & G/B(X,G) \\
X \arrow[right hook->]{u} \arrow[two heads]{r} & X/B(X,G) \arrow[right hook->]{u}
\end{tikzcd}
\]
where $X/B(X,G) \subset G/B(X,G)$ is a closed subvariety. 
\end{definition}

By the following result, the Ueno fibration of $X$ allows us to identify closed subvarieties of a semi-abelian variety which are of logarithmic general type. 

\begin{theorem}[\protect{\cite[Theorem 5.16]{Vojta:IntegralPointsII}}]\label{thm:KawamataUenofibration}
Let $X$ be a closed subvariety of a semi-abelian variety $G/K$, and let $D$ be an effective Weil divisor on $X$. Then, $B(X\setminus D, G) = 0$ if and only if $X\setminus D$ is of logarithmic general type. 
\end{theorem}

%
%

We can now prove \autoref{thm:closedabelianspecial}. 

\begin{proof}[Proof of \autoref{thm:closedabelianspecial}]
It is clear that if $X$ is the translate of a semi-abelian subvariety, then $X^{\an}$ is $K$-analytically special, so we need to prove the converse i.e., if $X^{\an}$ is $K$-analytically special then it is the translate of a sem-abelian subvariety.

Let $B(X,G)$ be the semi-abelian subvariety from \autoref{defn:stablizier} and \autoref{lemmastabilizersemiabelian}. 
By considering the Ueno fibration, we have a fibration $X \to X/B(X,G)$ where  $X/B(X,G)\subset G/B(X,G)$ is a closed subvariety of a semi-abelian variety. 
Now \autoref{thm:KawamataUenofibration} asserts that $X/B(X,G)$ is of logarithmic general type, and by \autoref{thm:starting_point}, we have that $(X/B(X,G))^{\an}$ is pseudo-$K$-analytically Brody hyperbolic. 
Moreover, \autoref{thm:nofibrationtoBrody} implies that $X/B(X,G) $ must be zero-dimensional, and hence $X$ is isomorphic to $B(X,G)$, which is the translate of a semi-abelian subvariety. 
\end{proof}

We record the following fact, which can be proven in exactly the same way as \autoref{thm:closedabelianspecial}, using Ueno fibration and Vojta's result. We note that the version for proper $G$ is \cite[Theorem 3.4]{javanpeykar2020albanese}.

\begin{theorem}\label{thm:specialsemiabelian}
Let $G/K$ be a semi-abelian variety and fix a smooth compactification $\overline{G}$ with boundary a normal crossing divisor. Let $X$ be a closed subvariety of $G$. Then $X$ is special, in the sense of \autoref{def:specialfibrationlog}, if and only if $X$ is a translate of a semi-abelian variety.
\end{theorem}

\begin{proof}[Proof of \autoref{coro:specialabelian}]
The proof follows immediately from \autoref{thm:closedabelianspecial} and \autoref{thm:specialsemiabelian}.
\end{proof}

Using \autoref{thm:closedabelianspecial}, we deduce that the quasi-Albanese map associated with a $K$-analytically special variety is surjective. 
First, we recall the definition of the quasi-Albanese variety associated with a smooth, quasi-projective variety. 

\begin{definition}[\cite{Serre_Albanese, iitaka_1977}]\label{defn:quasiAlb}
Let $V/K$ be a smooth variety. The \cdef{quasi-Albanese map} 
\[
\alpha\colon V \to \QAlb(V)
\]
is a morphism to a semi-abelian variety $\QAlb(V)$ such that:
\begin{enumerate}
\item For any other morphism $\beta\colon V \to B$ to a semi-abelian variety $B$, there is a morphism $f\colon \QAlb(V)\to B$ such that $\beta = f\circ \alpha$, and
\item the morphism $f$ is uniquely determined. 
\end{enumerate}
\end{definition}

\begin{prop}
Let $X/K$ be a smooth variety, and suppose that $X^{\an}/K$ is $K$-analytically special. Then, the quasi-Albanese map is surjective. 
\end{prop}

\begin{proof}
The image of $X$ in $\QAlb(X)$ is $K$-analytically special, as we can simply compose $\sU \to X^{\an}$ with the analytification of $\alpha$.
Moreover, by \autoref{thm:closedabelianspecial}, any $K$-analytically closed subvariety of $\QAlb(X)$ is the translate of a semi-abelian subvariety, which implies that the image of $X \to \QAlb(X)$ is the translate of a semi-abelian subvariety. 
The universal property of the quasi-Albanese (\autoref{defn:quasiAlb}) asserts that the image must equal $\QAlb(X)$.
\end{proof}

\section{\bf $K$-analytically special surfaces of negative Kodaira dimension}
\label{sec:sufacenegKodaira}

In this section, we characterize $K$-analytically special surfaces of Kodaira dimension $-\infty$ in terms of their irregularity and their tempered fundamental group. In particular, we will prove the following theorem.

\begin{theorem}[= Theorem \ref{thm:surfaceKinfinity}]
Let $K$ be an algebraically closed, complete, non-Archimedean valued field of characteristic zero. 
If $X/K$ is a smooth, projective surface with $\kappa(X) = -\infty$, then the following are equivalent:
\begin{enumerate}
\item $X$ has irregularity $q(X)$ less than 2;
\item $X$ is $K$-analytically special;
\item the tempered fundamental group $\pi_1^{\temp}(X^{\Ber})$ of $X^{\Ber}$ is virtually abelian.
\end{enumerate}
\end{theorem}

In this section we will be using Berkovich spaces rather than adic spaces, and we will write $X^{\Ber}$ to denote the Berkovich analytification of $X$. Thanks to \autoref{thm:comparision}, there is not much harm in doing so and we can use all the previously proven results. The choice of Berkovich spaces over adic spaces is due to the fact that we will be using the tempered fundamental group, a notion first introduced by Andr\'e \cite{andre-per} and whose definition we will recall below, and Berkovich spaces are better when one discusses topological coverings.

\subsection{The tempered fundamental group}
For this section, we say that a \cdef{$K$-manifold} is a connected, smooth, paracompact, strict Berkovich $K$-analytic space. 
We note that the analytification of a smooth variety $X/K$ will be a $K$-manifold by \autoref{lemma:absoluteproperties} and \cite[Theorem 3.4.8.(ii)]{BerkovichSpectral}. 
Let $f\colon \sX'\to \sX$ be a morphism of $K$-manifolds. 

First, we recall the notion of an \'etale covering from \cite{DeJongFundamentalGroupNonArch}. 

\begin{definition}\label{defn:etalecover}
We say that $f$ is an \cdef{\'etale covering} if $\sX$ is covered by open subsets $\sU$ such that $f^{-1}(\sU) = \sqcup \sV_j$ and $\sV_j\to \sU$ is finite \'etale. 
\end{definition}

We now distinguish between several types of \'etale coverings.

\begin{definition}
Let $f\colon \sX' \to \sX$ be an \'etale covering and keep the notation from \autoref{defn:etalecover}. 
\begin{itemize}
\item We say that $f$ is a \cdef{topological covering} if we can choose $\sU$ and $\sV_j$ such that all the maps $\sV_j\to \sU$ are isomorphisms.
\item We say that $f$ is a \cdef{finite \'etale covering} if $f$ is also finite.
\item We say that $f$ is \cdef{tempered} if it is the quotient of the composition of a topological covering $\sY'\to \sY$ and of a finite \'etale covering $\sY\to \sX$.  Here quotient means that we have a diagram 
\[
\xymatrix@R=1em{
& \sY'\ar[dr]\ar[dl]  &    \\
\sX'\ar[dr] &  & \sY\ar[dl] \\
& \sX & }
\]
or, equivalently, a tempered covering is an \'etale covering which becomes a topological covering after pullback by some finite \'etale covering. 
\end{itemize}
\end{definition}

Using the language of fiber functors, we can define the topological, algebraic, tempered, and \'etale fundamental group of a $K$-manifold (see \cite[Section 2]{DeJongFundamentalGroupNonArch} and \cite[Chapter 3, Section 2]{andre-per}). 
For a $K$-manifold $\sX$, we will let $\pi_1^{\topo}(\sX)$, $\pi_1^{\alg}(\sX)$, $\pi_1^{\temp}(\sX)$, and $\pi_1^{\ett}(\sX)$ denote each of these respective fundamental groups. 
We now give an important set of examples of tempered fundamental groups as well as a result concerning the birational invariance of the tempered fundamental group.

\begin{example}\label{exam:temperedsmallgenus}
The tempered fundamental group of the analytification of a smooth, projective curve $C/K$ of genus $g\leq 1$ is completely understood. 
Indeed, when $C \cong \mP^1$, then $\pi_1^{\temp}(C^{\Ber}) = \brk{e}$, and when $C$ is an elliptic curve, we have that $\pi_1^{\temp}(C^{\Ber})$ is either isomorphic to $\widehat{\mZ}^2$ or $\mZ \times \widehat{\mZ}$, depending on the reduction type of $E$ (see \cite[Chapter III, Section 2.3.2]{andre-per}). 
\end{example}

\begin{prop}\label{lem:birationalinvariancetemp}
Let $\sX,\sY$ be a smooth, proper $K$-manifolds over $K$ and let $\sX \dashrightarrow \sY$ be a bi-meromorphic morphism. Then $\pi_1^{\temp}(\sX) \cong \pi_1^{\temp}(\sY)$. 
\end{prop}

\begin{proof}
This is \cite[Proposition 1.5]{Lepage:TemperedFundamentalGroup}. 
\end{proof}

For our proof of \autoref{thm:surfaceKinfinity}, it will be key for us to understand when the tempered fundamental group of a curve is virtually abelian.  

\begin{lemma}\label{lemma:curvesvirtuallyabeliantemp}
Let $C/K$ be a smooth, projective curve. 
The tempered fundamental group $\pi_1^{\temp}(C^{\Ber})$ of $C^{\Ber}$ is virtually abelian if and only if the genus of $C$ is less than two.
\end{lemma}

\begin{proof}
When the genus of $C$ is less than one, $\pi_1^{\temp}(C^{\Ber})$ is virtually abelian by \autoref{exam:temperedsmallgenus}.  

Conversely, we need to show that when $C$ has genus greater than one, $\pi_1^{\temp}(C^{\Ber})$ is not virtually abelian. 
First, we note that $\pi_1^{\temp}(C^{\Ber})$ is centerless, and hence non-abelian. 
This follows from noting that the profinite completion of $\pi_1^{\temp}(C^{\Ber})$ is the algebraic fundamental group $\pi_1^{\alg}(C)$ (see \cite[p.~128, paragraph 2]{andre-per}), and the algebraic fundamental group of a curve is well-known to be centerless (see e.g., \cite[Lemma 1]{Faltings:CurvesBourbaki}). 
To conclude, we note that if $\pi_1^{\temp}(C^{\Ber})$ was virtually abelian, then there would exist a finite \'etale covering $\sC'\to C^{\Ber}$ whose tempered fundamental group is abelian. 
Note that $\sC'$ is algebraic by non-Archimedean Riemann existence theorem \cite{Lutkebohmert:Riemannexistence}, and so there exists some smooth, projective curve $C'$ such that $C^{'\Ber}\cong \sC'$. 
The genus of $C'$ is strictly larger than the genus of $C$, and hence the tempered fundamental group of $C^{'\Ber}$ cannot be abelian by the above discussion. 
Therefore, we have that if $C$ has genus greater than one, $\pi_1^{\temp}(C^{\Ber})$ is not virtually abelian. 
\end{proof}

\begin{proof}[Proof of \autoref{thm:surfaceKinfinity}]
By \autoref{lemma:birationalinvariance} and \autoref{lem:birationalinvariancetemp}, all of the proposed equivalent conditions in the statement of \autoref{thm:surfaceKinfinity} are birational invariants so we are free to make birational modifications. 
By the Enriques--Kodaira classification theorem (see e.g., \cite[Chapter V, Theorem 6.1]{Har77}), we know that a smooth, projective surface $X$ of Kodaira dimension $-\infty$ is birational to a $\mP^1$-bundle over a curve $C$ over $K$.

Since $\mP^1$ admits no analytic differentials and is simply connected in either the topological or tempered sense (see \autoref{exam:temperedsmallgenus}), we have that 
\[
q(X) = q(C), \quad \pi_1^{\topo}(X^{\Ber}) \cong \pi_1^{\topo}(C^{\Ber}), \text{ and }\quad\pi_1^{\temp}(X^{\Ber}) \cong \pi_1^{\temp}(C^{\Ber}). 
\]
To prove our result, we will show that having $q(X) > 1$ is equivalent to $\pi_1^{\temp}(X^{\Ber})$ not being virtually abelian and to $X$ not being  $K$-analytically special. 
By the above, we have that $q(X) > 1$ if and only if $C$ has genus greater than one, and so the first statement follows from \autoref{lemma:curvesvirtuallyabeliantemp}.

The second statement follows because a curve $C$ of genus greater than 1 is $K$-analytically Brody hyperbolic \cite[Proposition 3.15]{JVez}. 
Moreover, any analytic morphism from a big analytic open of the analytification of a connected, finite type algebraic group will be constant in $C^{\Ber}$ by \autoref{rem:Brodymoduloempty}, and hence cannot be Zariski dense in $X^{\Ber}$ as it is contained in a fiber of the morphism $X \to C$. 
Note that if $C$ has genus less than one, then $X$ is birational to either $\mP^1 \times E$ where $E/K$ is an elliptic curve or $\mP^1 \times \mP^1$, which are both clearly $K$-analytically special. 
\end{proof}

\begin{remark}
The proof of \autoref{thm:surfaceKinfinity} tells us that one cannot replace the tempered fundamental group with the topological fundamental group. 
To see this, it suffices to note that a curve $C/K$ of genus $g\geq 2$ with good reduction will have trivial topological fundamental group, which of course is abelian. 
We will use this observation when formulating our non-Archimedean counterparts to Campana's conjectures in Section \ref{sec:NonArchCampana}.
\end{remark}

\section{\bf Non-Archimedean variants of Campana's conjectures}
\label{sec:NonArchCampana}
In this final section,  based on our results above, we formulate two non-Archimedean variants of a series of conjectures of Campana concerning various notions of specialness and their relationship to fundamental groups.

\begin{conjecture}[Campana's conjectures, extended to include $k$-analytically special]\label{conjecture:Campana}
Let $X/k$ be a smooth projective variety. The conditions (1) --- ($4_{\infty}$) and (1) --- ($4_{p}$) are equivalent:
\begin{enumerate}
\item $X$ is special (\autoref{def:specialBogomolov}, \autoref{def:specialfibration});
\item $X$ is geometrically special (\autoref{defn:geometricallyspecial});
\item $X$ is arithmetically special (\autoref{defn:arithmeticallyspecial});
\item[($4_{\infty}$)] if $k = \mC$, $X$ is Brody special (\autoref{defn:Brodyspecial});
\item[($4_{p}$)] if $k$ is a complete, non-Archimedean valued field, $X$ is $k$-analytically special (\autoref{defn:Kanspecial}). 
\end{enumerate}
\end{conjecture}

\begin{conjecture}[Campana's abelianity conjecture for fundamental groups, extended to include $k$-analytically special]\label{conjecture:Abelian}
Let $X/k$ be a smooth projective variety. The conditions (1) --- ($4_{\infty}$) and (1) --- ($4_{p}$) are equivalent:
\begin{enumerate}
\item If $k = \mC$ and $X$ is special, then $\pi_1(X^{\an})$ is virtually abelian.
\item If $k = \mC$ and $X$ is Brody special, then $\pi_1(X^{\an})$ is virtually abelian.
\item If $k = \mC$ and $X$ is geometrically special, then $\pi_1(X^{\an})$ is virtually abelian.
\item[($4_{\infty}$)] If $k = \mC$ and $X$ is arithmetically special, then $\pi_1(X^{\an})$ is virtually abelian.
\item[($4_p$)] If $k$ is a complete, non-Archimedean valued field and $X^{\Ber}$ is $k$-analytically special, then $\pi_1^{\temp}(X^{\Ber})$ is virtually abelian.
\end{enumerate}
\end{conjecture}

\begin{remark}
The complex analytic part of \autoref{conjecture:Campana}, namely the equivalence of conditions (1) and ($4_{\infty}$), has equivalent incarnation via the Kobayashi pseudo-metric. In \cite[Conjecture 9.2]{campana2004orbifolds}, Campana asked if for a proper variety $X/\mC$, $X$ being special is equivalent to the Kobayashi pseudo-metric $d_X$ vanishing identically on $X^{\an}$. 
For a discussion on the Kobayashi pseudo-metric and its relation to hyperbolicity, we refer the reader to \cite{Kobayashi}. 
Cherry \cite{CherryKoba} offered a non-Archimedean variant of the Kobayashi pseudo-metric, however his notion seems to not correctly capture hyperbolic or special properties of a variety. 
For example, he proves (\textit{loc.~cit.~}Theorem 4.6) that his non-Archimedean Kobayashi pseudo-metric is a genuine metric on an abelian variety. 

In forthcoming work \cite{morrow:NonArchKob}, the first author provides a new definition of a non-Archimedean Kobayashi pseudo-metric, which does seem to correctly capture hyperbolic and special properties of a variety. 
As an illustration of this claim, the author shows that for $K$ an algebraically closed, complete, non-Archimedean valued field of characteristic zero that contains a countable dense subset and $\sX$ a good, connected Berkovich $K$-analytic space, if this new non-Archimedean Kobayashi pseudo-metric $d_{\sX}$ is in fact a metric, then $\sX$ is $K$-analytically Brody hyperbolic in the sense of in the sense of \cite[Definition 2.3]{JVez}. 
With this result and others, it seems natural to conjecture that for such a $K$, a Berkovich $K$-analytic space is $K$-analytically special if and only if this new non-Archimedean Kobayashi pseudo-metric $d_{\sX}$ is identically zero on $\sX$. 
\end{remark}

\appendix

  \bibliography{refs}{}
\bibliographystyle{amsalpha}

 \end{document}